\documentclass[reqno,english,11pt]{amsart}
\usepackage[margin=1in]{geometry}
\usepackage[latin9]{inputenc}
\usepackage{wasysym}
\usepackage[colorlinks=true]{hyperref}
\usepackage{textcomp}
\usepackage{amstext}
\usepackage{amsthm}
\usepackage{amssymb}
\usepackage{hyperref}
\usepackage{cite}
\usepackage{url}
\usepackage{xcolor}
\makeatletter

\newtheorem{theorem}{Theorem}[section]
\newtheorem{lemma}[theorem]{Lemma}
\newtheorem{proposition}[theorem]{Proposition}
\newtheorem{corollary}[theorem]{Corollary}
\theoremstyle{remark}
\newtheorem{observation}[theorem]{Remark}

\numberwithin{equation}{section}

\newcommand{\be}{\begin{equation}}
\newcommand{\ee}{\end{equation}}
\newcommand{\lb}{\label}
\newcommand{\dd}{\,d}
\newcommand{\les}{\lesssim}
\newcommand{\ges}{\gtrsim}
\newcommand{\ov}{\overline}
\newcommand{\R}{\mathbb R}
\newcommand{\Z}{\mathbb Z}

\newcommand{\B}{\mathfrak B}
\renewcommand{\S}{\mathcal S}
\newcommand{\mc}{\mathcal}
\DeclareMathOperator{\re}{Re}
\renewcommand{\Re}{\re}
\DeclareMathOperator{\im}{Im}

\DeclareMathOperator{\sgn}{sign}
\DeclareMathOperator{\Res}{Res}

\title[Strichartz estimates for Klein--Gordon]{Strichartz estimates for the Klein--Gordon equation in $\R^{3+1}$}
\author{Marius Beceanu}
\address{University at Albany SUNY, Department of Mathematics and Statistics, Earth Science 110, Albany, NY, 12222, USA}
\email{mbeceanu@albany.edu}
\author{Gong Chen}
\address{University of Toronto, Department of Mathematics, 40 St George Street Rm 6290, Toronto, ON M5S 2E4, Canada}
\email{gc@math.toronto.edu}
\date{\today}
\thanks{The first author was partially supported by NSF grant DMS-1700293 during the preparation
of this work. The first author  thanks the University of Chicago for its hospitality during the summer of
2017.}
\thanks{The second author thanks Department of Mathematics and Statistics, University at Albany SUNY  for its hospitality during the springs of
2018 and 2019.}
\usepackage{mathtools}

\begin{document}
	\maketitle
	
	\begin{abstract} In this paper we prove 	 standard and reversed Strichartz estimates for the Klein--Gordon equation in $\R^{3+1}$.  Instead of the Fourier theory, our analysis is based on  fundamental solutions of the free equations and  fractional integrations. In the final part of this paper,  we apply Strichartz estimates in the study of a semilinear Klein--Gordon equation.
	\end{abstract}
	
	\tableofcontents

	\section{Introduction}
	\subsection{The linear Klein--Gordon equation}
	In this paper we establish some new properties of the Klein--Gordon equation in $\R^{d+1}$, 
	\begin{equation}
	u_{tt} - \Delta u + m^2 u + V u = F,\ u(0)=u_0,\ u_t(0)=u_1.\lb{kg}
	\end{equation}
	We shall focus on the three-dimensional case $d=3$. Equation (\ref{kg}) is one of the best-known equations of mathematical physics, describing the behavior of a relativistic particle (boson) of rest mass $m$ and spin zero. Here $V$ is a scalar potential; we shall call the equation free if $V=0$ and perturbed otherwise.
	
	Equation (\ref{kg}) is implied by Dirac's equation, just as the wave equation is implied by Maxwell's equations. It reduces to the wave equation in the case $m=0$, but has different properties from the latter if $m \ne 0$. Assuming $m \ne 0$, with no loss of generality we shall take $m=1$.

		In this paper we will prove several dispersive and decay estimates for the linear Klein--Gordon equation in $\R^{3+1}$, both with and without potentials, inspired by the work Beceanu-Goldberg \cite{becgol} for wave equations. Then we will illustrate some nonlinear applications of  these estimates to semilinear problems.

	\subsection{The free equation}
	
	The free inhomogenous linear Klein--Gordon equation in $\R^{d+1}$ is
	\be\lb{fkg}
	u_{tt} - \Delta u + u = F,\ u(0)=u_0,\ u_t(0)=u_1.
	\ee
	With a time-independent scalar potential, the equation becomes
	\be\lb{pkg}
	u_{tt} - \Delta u + u + V(x) u = F,\ u(0)=u_0,\ u_t(0)=u_1.
	\ee
	The Klein--Gordon equation is not scaling-invariant, behaving either like the wave or like the Schr\"{o}dinger equation in different regimes.

	The solution to equation (\ref{fkg}) is given by the Duhamel formula
	\be\lb{duhamel}
	u(t) = \cos(t \sqrt{-\Delta+1}) u_0 + \frac {\sin(t\sqrt{-\Delta+1})}{\sqrt{-\Delta+1}} u_1 + \int_0^t \frac {\sin((t-s)\sqrt{-\Delta+1})}{\sqrt{-\Delta+1}} F(s) \dd s.
	\ee
	
	The sine propagator for the three-dimensional Klein--Gordon equation has the kernel
	\begin{equation}\label{eq:kgsplitfree}
	\frac {\sin(t\sqrt{-\Delta+1})}{\sqrt{-\Delta+1}} = \frac {\sgn t} {4\pi r} \delta_r(\tau) - \frac {\sgn t} {4\pi} \chi_{|t| \geq r} \frac {J_1(\sqrt{t^2-r^2})} {\sqrt{t^2-r^2}},
	\end{equation}
	where $r=|x|$.
	
	Note that it is made of two parts, namely the free wave propagator plus an extra term specific to the Klein--Gordon equation. The latter oscillates, with a leading behavior of
	$$
	\frac {\sin \sqrt{t^2-r^2}}{(t^2-r^2)^{3/4}}.
	$$
	The problem is that the $L^1_t$ norm of this component is only of size $r^{-1/2}$, which is different from the signed measure norm of the free wave part:
	$$
	\bigg\|\frac {\sin \sqrt{t^2-r^2}}{(t^2-r^2)^{3/4}}\bigg\|_{L^1_t} \sim r^{-1/2},\ \Big\|\frac 1 r \delta_r(t)\Big\|_{\mc M_t} \sim \frac 1 r.
	$$
	Thus, the sine propagator appears to have not one specific scaling, but several. This intuition is in fact confirmed by our computations.
	
	\subsection{The perturbed equation}
	The Hamiltonian of the free equation, $-\Delta+1$, consists of two terms that scale differently, namely $-\Delta$ and $1$. Thus, the scalar potential should scale like either or have some intermediate behavior.
	
	Hence, in $\R^{d+1}$ it is natural to consider potentials $V \in L^p$, $d/2 \leq p \leq \infty$, or more generally in the corresponding Lorentz spaces; see Bergh--L\"{o}fstr\"{o}m \cite{bergh} for the definition and properties of the latter.
	
	In this paper we consider real-valued, short-range scalar potentials $V \in L^{3/2, 1}$ on $\R^3$. Then $H$ has no positive energy bound states and at most discretely many negative energy bound states (which can, however, accumulate toward the edge of the essential spectrum).

		The properties of the solution to the equation \eqref{pkg} largely depend on those of the Hamiltonian $H=-\Delta+V$. Exponentially growing modes correspond to eigenvalues below $-1$, while eigenvalues between $-1$ and $0$ give rise to oscillating modes that neither grow nor disperse. Only the continuous spectrum portion of the solution is dispersive.  Obtaining the optimal dispersion rate also requires the absence of threshold bound states, since these disperse at a different rate, see Beceanu \cite{Bec} for the Schr\"odinger problem.

	Henceforth we shall assume that $H=-\Delta+V$ has no zero energy bound states. This also simplifies the analysis of the point spectrum, guaranteeing at most finitely many eigenvalues.
	
	\subsection{Main results}
	
	It is well-known that one important tool in understanding both the linear and the nonlinear equations (\ref{kg}) and (\ref{kgnl}) below consists in Strichartz estimates -- evaluating the solution in mixed spacetime Sobolev norms.	Our main results in this paper are standard Strichartz estimates which measure the solution in the $L^p_t L^q_x$ norm and as well as reversed Strichartz estimates involving $L^q_x L^p_t$ norms.

	In order to analyze the Klein--Gordon evolution operator, as in the free setting, \eqref{eq:kgsplitfree}, we split it into two parts, the wave evolution operator and the ``Bessel'' part, the difference between the two.
	
	These two operators have different properties. The wave evolution operator was studied in \cite{becgol}. Our main result for the ``Bessel'' part of the evolution is as follows:
	\begin{theorem}[Theorem \ref{SBH}] Consider the Klein--Gordon equation in $\R^{3+1}$ with Hamiltonian $H=-\Delta+V$, $V \in L^{3/2, 1}$, such that $H$ has no threshold bound states (eigenstates or resonance \footnote{ Recall that $\psi$
			is a resonance of $H$ at $0$ if it is a distributional solution
			of the equation $H\psi=0$ which belongs to the space $L^{2}\left(\left\langle x\right\rangle ^{-\sigma}dx\right):=\left\{ f:\,\left\langle x\right\rangle ^{-\sigma}f\in L^{2}\right\} $
			for any $\sigma>\frac{1}{2}$, but not for $\sigma=\frac{1}{2}.$}). Then the difference of kernels (the Bessel part)
		$$
		S_B^H(t) = \frac {\sin(t\sqrt{H+1})P_c}{\sqrt{H+1}} - \frac {\sin(t\sqrt{H})P_c}{\sqrt{H}}
		$$
		is bounded in $L^\infty_{x, y, t}$ and
		$$
		\|S_B^H(t)\|_{L^1_t} \les |x-y|^{-1/2},\ \|S_B^H(t)\|_{L^{4/3, \infty}_t} \les |x-y|^{-3/4},\ \|S_B^H(t)\|_{L^\infty_t} \les 1.
		$$
		Consequently $S_B^H$ is dominated by a convolution operator, with convolution kernel in $L^{6, \infty}_x L^1_t$, $L^{4, \infty}_x L^{4/3, \infty}_t$, and $L^\infty_{x, t}$.
	\end{theorem}
	
	In addition, we obtain a number of reversed Strichartz estimates for the perturbed Klein--Gordon evolution, the most relevant of which are the following:
	\begin{proposition}[Reversed Strichartz estimates, see Corollary \ref{cor_partial_results}] Consider the Klein--Gordon equation in $\R^{3+1}$ with Hamiltonian $H=-\Delta+V$, $V \in L^{3/2, 1}$, such that $H$ has no threshold bound states (eigenstates or resonance). Then the following operators are bounded:
		$$
		\frac {e^{it\sqrt{H}} P_c}{\sqrt H}: L^2 \to L^\infty_x L^2_t \cap L^{12, 2}_x L^2_t \cap L^{6, 2}_x L^\infty_t \cap L^{24/5, 2}_x L^{8, 2}_t.
		$$
	\end{proposition}
	After computing the pointwise decay of fractional integrations of the sine propagator, via  the $TT^*$ argument, we also obtain standard Strichartz estimates.
	\begin{theorem}[Strichartz estimates, see Theorem 	\ref{thm:KGStrichartz}]
		\label{thm:KGStrichartzintro} Consider the Klein--Gordon equation in $\R^{3+1}$ with Hamiltonian $H=-\Delta+V$, $V \in L^{3/2, 1}$, such that $H$ has no threshold bound states (eigenstates or resonance). Then for 
		\begin{equation}
		\frac{2}{p}+\frac{3}{q}=\frac{3}{2},\,p\geq2\label{eq:661int}
		\end{equation}
		one has the following Strichartz estimates
		\begin{equation}
		\left\Vert E_{\frac{1}{2}\left(\frac{1}{p}+\frac{1}{2}-\frac{1}{q}\right)}^{H}\left(t\right)P_{c}f\right\Vert _{L_{t}^{p}L_{x}^{q}}\lesssim\left\Vert f\right\Vert _{L_{x}^{2}}\label{eq:662int}
		\end{equation}
		where  we used the notation $E_{\alpha}^{H}(t)=\frac{e^{it\sqrt{H+1}}}{\left(H+1\right)^{\alpha}}$.
	\end{theorem}
	\subsection{The semilinear Klein--Gordon equation} Another model of physical interest is obtained by adding a nonlinear monomial potential
	\be\lb{kgnl}
	u_{tt} - \Delta u + m^2 u + V u \pm u^{N+1} = F,\ u(0)=u_0,\ u_t(0)=u_1.
	\ee
	Usually $N$ is taken to be an integer, $2$ or $4$ in particular, but other values are also allowed, as well as more general nonlinearities.
	
	Using linear estimates above, we then illustrate several nonlinear applications, proving that the equation (\ref{kgnl}) is globally well-posed for small data in these norms. For large data, using the trick that choping the data into pieces, then using finite speed of propagation to patch each piece together, we can obtain the local well-posedness for large data from the small data results.
	
	All throughout we consider mild solutions to (\ref{kgnl}), i.e.\;solutions in the sense of semigroups:
	$$\begin{aligned}
	u(t) &= \cos(t\sqrt{H+1}) u_0 + \frac {\sin(t\sqrt{H+1})} {\sqrt{H+1}} u_1 \\
	&\mp \int_0^t \frac {\sin((t-s)\sqrt{H+1})} {\sqrt{H+1}} (u^{N+1}(s) u(s)+F(s)) \dd s.
	\end{aligned}$$
	Here $H=-\Delta+V$ is the Hamiltonian of the corresponding wave equation.
	
	Our main nonlinear result is the following application of reversed Strichartz estimates:
	
	\begin{proposition}[Proposition \ref{prop_nonlinear}] Consider a potential $V \in L^{3/2, 1}(\R^3)$ such that $-\Delta+V$ has no negative eigenvalues and no threshold eigenfunctions or resonance. The quintic Klein--Gordon equation in $\R^{3+1}$
		$$
		u_{tt}-\Delta u + Vu + u \pm u^5 = F,\ u(0)=v,\ u_t(0)=w
		$$
		is globally well-posed if the norm of the initial data $\|(v, w)\|_{H^1 \times L^2}$ is sufficiently small, if $F$ is also small in the norm of
		$
		L^{6/5, 2}_x L^\infty_t \cap L^{16/15, 2}_x L^{16/5, 2}_t \cap L^{12/11, 2}_x L^{4, 1}_t \cap L^{48/41, 2}_x L^{16/3, 2}_t.
		$
		
		The solution $u$ is in $L^{6, 2}_x L^\infty_t \cap L^{16/3, 2}_x L^{16, 2}_t$ and
		$$
		\|u\|_{L^{6, 2}_x L^\infty_t \cap L^{16/3, 2}_x L^{16, 2}_t} \les \|(v, w)\|_{H^1 \times L^2} + \|F\|_{L^{6/5, 2}_x L^\infty_t \cap L^{16/15, 2}_x L^{16/5, 2}_t \cap L^{12/11, 2}_x L^{4, 1}_t \cap L^{48/41, 2}_x L^{16/3, 2}_t}.
		$$
		
		If in addition $F=0$, then $u \in L^\infty_x L^2_t \cap L^{12, 2}_x L^2_t \cap L^{24/5, 2}_x L^{8, 2}_t$ as well.
	\end{proposition}
	
	%
	%
	\subsection{Discussion of the results}

	The study of standard Strichartz estimates for the Klein--Gordon equation has a long history, see for example D'Ancona--Fanelli \cite{DaF}, Beals--Strauss \cite{BS}, Keel--Tao \cite{KeTa} and Nakanishi--Schlag \cite{NaSch}.
	
	\textbf{Reversed Strichartz estimates.}
	In $\mathbb{R}^{3}$, it is known that the endpoint Strichartz estimate
	$L_{t}^{2}L_{x}^{\infty}$ fails both for the wave and Klein-Gordon
	equations, see for example  Machihara-Nakamura-Nakanishi-Ozawa \cite{MNNO}. But in many nonlinear stability
	analysis problems, for example in Chen \cite{C2} and Jia-Liu-Schlag-Xu \cite{JLSchX}, one has to deal with the
	quadratic terms. Then the reversed type estimates $L_{x}^{\infty}L_{t}^{2}$
	play a crucial role in the analysis.
	
	Except for the homogeneous
	Strichartz estimate, the standard Strichartz estimates are not Lorentz
	invariant. But using the reversed type estimates, one can obtain some
	Lorentz invariance using a family of reversed type estimates which
	will be important in the analysis of multisoliton, see Chen \cite{C1,C2}.
	We should point out that these reversed type estimates for perturbed
	equations do not follow from the $L^{p}$ boundedness of wave operators.
	
	Under some stronger decay assumptions on the potential $V$, one can infer some of our estimates by using the wave operator structure formula of  Beceanu--Schlag \cite{BeSch}. However, here we prove them directly, under less stringent assumptions on the potential and for a wider range of exponents, in the spirit of the work of  Beceanu--Goldberg \cite{becgol}.

	\textbf{Standard Strichartz estimates.} The study of standard Strichartz estimates for the Klein--Gordon equation has a long history, see for example D'Ancona--Fanelli \cite{DaF}, Beals--Strauss \cite{BS}, Ibrahim--Masmoudi--Nakanishi \cite{IMN}, Keel--Tao \cite{KeTa}, Machihara--Nakamura--Nakanishi--Ozawa \cite{MNNO2} and Nakanishi--Schlag \cite{NaSch}.
	
	The following is a complete list of the standard Strichartz estimates for the Klein--Gordon equation in three dimensions, as per \cite{DaF}:
	\begin{proposition} The free Klein--Gordon flow in three dimensions $e^{it\sqrt{-\Delta+1}}$ satisfies the Strichartz estimates
		$$
		\|e^{it\sqrt{-\Delta+1}} f\|_{L^p_t W^{1/q-1/p-1/2, q}_x} \les \|f\|_{L^2_x}
		$$
		whenever $2 \leq p, q \leq \infty$, $\frac 2 p + \frac 3 q = \frac 3 2$ (Schr\"{o}dinger-admissible exponents).
	\end{proposition}
	
	All these norms can be obtained by complex interpolation between the trivial endpoint $L^\infty_t L^2_x$ and the nontrivial endpoint $L^2_t W^{-5/6, 6}_x$. By the Sobolev embedding, this nontrivial endpoint implies an $L^2_t L^9_x$ bound for the sine propagator $S_{1/2}$, see \cite{NaSch}, and is in turn equivalent, by the $T T^*$ method, to the claim that
	$$
	C_{5/6}(t)=\frac {\cos(t\sqrt{-\Delta+1})}{(-\Delta+1)^{5/6}} \in \B(L^2_t L^{6/5}_x, L^2_t L^6_x).
	$$
	
	Many proofs of Strichartz estimates for wave and Klein--Gordon equations involve a Littlewood-Paley dyadic decomposition in frequency. However, the Littlewood--Paley theory for the perturbed Hamiltonian is more involved, see Beceanu-Goldberg \cite{becgol1} and references therein. Since the perturbed equation is not scaling-invariant, unlike in the free case. Thus, we try to avoid it in the current paper. For non-endpoint estimates, our approach is inspired by Beals \cite{BE}. We do not rely on a dyadic decomposition in the frequency space. Instead, we use complex interpolation between pointwise decay estimates.
	
	The endpoint Strichartz estimates, see Keel--Tao \cite{KeTa}, are of crucial importance in the analysis
	of dispersive PDE. In their classic paper, Keel and Tao used real interpolation between a two-parameter family of bilinear estimates to prove the endpoint estimate in a general framework. For the endpoint Strichartz estimates, we have to invoke Keel--Tao \cite{KeTa}  approach with the Littlewood-Paley dyadic decomposition in frequency developed by Beceanu-Goldberg \cite{becgol1}.
	
	%
	%
	
	%
	%
	%
	
	\subsection{Notations}\lb{notation}
	We denote the Klein--Gordon propagators that we study in this paper as follows:
	$$
	S_\alpha(t) = \frac {\sin(t\sqrt{-\Delta+1})}{(-\Delta+1)^\alpha},\ C_\alpha(t) = \frac {\cos(t\sqrt{-\Delta+1})}{(-\Delta+1)^\alpha},\ E_\alpha(t) = \frac {e^{it\sqrt{-\Delta+1}}}{(-\Delta+1)^\alpha}.
	$$
	For the perturbed Hamiltonian, $H=-\Delta+V$, we also define:
	\[
	S_{\alpha}^{H}(t)=\frac{\sin\left(t\sqrt{H+1}\right)}{\left(H+1\right)^{\alpha}},\,C_{\alpha}^{H}(t)=\frac{\cos\left(t\sqrt{H+1}\right)}{\left(H+1\right)^{\alpha}},\,E_{\alpha}^{H}(t)=\frac{e^{it\sqrt{H+1}}}{\left(H+1\right)^{\alpha}}.
	\]
	\subsection{Organization}
	
	This paper is organized as follows: In Section \ref{FKG}, we
	analyze the free Klein-Gordon equation including the pointwise decay,
	reversed Strichartz estimates, fractional integration
	of the Klein-Gordon propagator and local energy decay. In Section
	\ref{KGP}, we extend the estimates for the free equation to the Klein-Gordon
	equation with a scalar potential,  including standard Strichartz estimates.
	In Section \ref{sec:nonlinear}, we illustrate some
	applications of the estimates we obtained from previous sections by
	showing the existence of the solutions to some semilinear Klein-Gordon
	equation with monomial nonlinearity. In Appendix \ref{sec:bessel}, we recall
	some facts about Bessel and Hankel functions. In Appendix \ref{sec:agmon}, we present the sharp Agmon estimates for the sake of completeness. Then finally, in Appendix \ref{sec:sob}, some comparisons between Sobolev spaces defined by the free Laplacian and the perturbed Hamiltonian are established.
	\smallskip
	
	\section{The free Klein--Gordon kernels}\lb{FKG}
	In this section, we analyze the kernels associated to the free Klein--Gordon equation. Reversed Strichartz estimates, pointwise decay and some local energy decay will be established. Notice that in this section, the physical representation of the fundamental solution to the linear Klein--Gordon equation plays a pivotal  role.
	\subsection{The cosine kernel} We first estimate the kernel of
	\be\lb{C_1}
	C_1=\frac {\cos(t\sqrt{-\Delta+1})}{-\Delta+1}.
	\ee
	For $t \geq 0$, one can write
	$$
	\frac {\cos(t\sqrt{-\Delta+1})}{-\Delta+1} = \int_t^\infty \frac {\sin(\tau\sqrt{-\Delta+1})}{\sqrt{-\Delta+1}} \dd \tau.
	$$
	Recall that for $r=|x|$
	$$
	S_{1/2}(r, \tau) = \frac {\sin(\tau\sqrt{-\Delta+1})}{\sqrt{-\Delta+1}} = \frac 1 {4\pi r} \delta_r(\tau) - \frac 1 {4\pi} \chi_{\tau \geq r} \frac {J_1(\sqrt{\tau^2-r^2})} {\sqrt{\tau^2-r^2}}.
	$$
	
	For each individual $x$, the domain of integration in (\ref{C_1}) is $\tau \in [\max(r, t), \infty)$. We distinguish two cases: $t \leq r$ and $t>r$. In the first case, $C_1$ does not depend on $t$ and we get a contribution from the wave equation term,
	$$
	\int_t^\infty \frac 1 {4\pi r} \delta_r(\tau) \dd \tau = \frac 1 {4\pi r},
	$$
	while in the second case this term does not contribute. Hence, the kernel has a jump discontinuity of size $1/r$ along the light cone $|t|=r$, but is otherwise continuous.
	
	These are our cosine kernel bounds:
	\begin{proposition}\lb{cosine_kernel} For $t>r$
		\be\lb{C1}
		|C_1(r, t)| = \bigg|\frac {\cos(t\sqrt{-\Delta+1})}{-\Delta+1}\bigg|(r, t) \les \frac 1 {t \langle t^2-r^2 \rangle^{1/4}}
		\ee
		and in addition it is uniformly bounded for $r<1$: $|C_1(r, t)| \les 1$.
		
		For $0 \leq t < r$ the kernel has exponential decay: for $r>1$
		\be\lb{C2}
		|C_1(r, t)| = \bigg|\frac {\cos(t\sqrt{-\Delta+1})}{-\Delta+1}\bigg|(r, t) \les \frac {e^{-r}}{r^{1/2}}
		\ee
		and an $r^{-1}$ bound holds for $0 \leq t < r < 1$: $|C_1(r, t)| \les r^{-1}$.
	\end{proposition}
	
	\begin{proof}
		
		Let $\tau=\sigma+r$. We make the change of variable
		$$
		\tilde \sigma(\sigma) = \sqrt{\sigma^2+2r\sigma} = \sqrt{\tau^2-r^2},
		$$
		which is an increasing, continuous, invertible map $\tilde \sigma(\sigma):[0, \infty) \to [0, \infty)$, so its inverse is also increasing and continuous.
		
		Explicitly, the inverse is given by
		\be\lb{inverse}
		\sigma(\tilde \sigma) = \sqrt{\tilde \sigma^2+r^2}-r,\ \sigma'(\tilde \sigma) = \frac {\tilde \sigma}{\sqrt{\tilde \sigma^2+r^2}}.
		\ee
		In particular, note that $\sigma'(\tilde \sigma) \geq 0$ for $\tilde \sigma \geq 0$. The following identity is useful:
		\be\lb{inverse'}
		\frac {d \sigma}{\tilde \sigma} = \frac {d \tilde \sigma}{r+\sigma(\tilde \sigma)} = \frac {d \tilde \sigma}{\sqrt{\tilde \sigma^2+r^2}}.
		\ee
		
		\noindent\textbf{I. Case $0 \leq t \leq r$.} For $0 \leq t \leq r$, this leads to
		\be\lb{integral}\begin{aligned}
			\int_r^\infty \frac {J_1(\sqrt{\tau^2-r^2})} {\sqrt{\tau^2-r^2}} \dd \tau &= \int_0^\infty \frac {J_1(\sqrt{\sigma^2+2r\sigma})} {\sqrt{\sigma^2+2r\sigma}} \dd \sigma = \int_0^\infty \frac {J_1(\tilde \sigma)} {\tilde \sigma} \dd \sigma\\
			&= \int_0^\infty \frac {J_1(\tilde \sigma)} {r+\sigma(\tilde \sigma)} \dd \tilde \sigma = \int_0^\infty \frac {J_1(\tilde \sigma)} {\sqrt{r^2+\tilde \sigma^2}} \dd \tilde \sigma\\
			&= \int_0^\infty \bigg(\int_{\tilde \sigma}^\infty \frac {\sigma'(\sigma^*) \dd \sigma^*} {(r+\sigma(\sigma^*))^2}\bigg) J_1(\tilde \sigma) \dd \tilde \sigma \\
			&= \int_0^\infty \frac {\sigma'(\tilde \sigma)}{(r+\sigma(\tilde \sigma))^2} \bigg(\int_0^{\tilde \sigma} J_1(\sigma^*) \dd \sigma^*\bigg) \dd \tilde \sigma \\
			&\leq \frac 1 r \sup_{\tilde \sigma \geq 0} \bigg|\int_0^{\tilde \sigma} J_1(\sigma^*) \dd \sigma^*\bigg| \leq \frac 2 r.
		\end{aligned}\ee
		
		Since this improper integral is uniformly bounded, $(\ref{integral}) \les C/r$.
		
		Also, we have an $r^{-1/2}$ bound: since $J_1(\tilde \sigma) \les \tilde \sigma^{-1/2}$,
		$$
		(\ref{integral}) = \int_0^\infty \frac {J_1(\tilde \sigma)} {\sqrt{r^2+\tilde \sigma^2}} \dd \tilde \sigma \les r^{-1/2}.
		$$
		One can also prove a uniform bound for small $r$, see at the end of this proof. However, such a bound is not important here, because we still have to subtract the free wave term.
		
		Taking into account the free wave term, there is more cancellation and a different bound is possible: due to formula (\ref{intj}),
		$$
		\frac 1 r = \int_0^\infty \frac {\sigma'(\tilde \sigma)}{(r+\sigma(\tilde \sigma))^2} \dd \tilde \sigma = \int_0^\infty \frac {\sigma'(\tilde \sigma)}{(r+\sigma(\tilde \sigma))^2} \bigg(\int_0^\infty J_1(\sigma^*) \dd \sigma^*\bigg) \dd \tilde \sigma,
		$$
		so
		\be\begin{aligned}\lb{temp1}
			\frac 1 r -\int_r^\infty \frac {J_1(\sqrt{\tau^2-r^2})} {\sqrt{\tau^2-r^2}} \dd \tau &= \int_0^\infty \frac {\sigma'(\tilde \sigma)}{(r+\sigma(\tilde \sigma))^2} \bigg(\int_{\tilde \sigma}^\infty J_1(\sigma^*) \dd \sigma^*\bigg) \dd \tilde \sigma \\
			&= \int_0^\infty \frac {\tilde \sigma J_0(\tilde \sigma)}{(r^2+\tilde \sigma^2)^{3/2}} \dd \tilde \sigma.
		\end{aligned}\ee
		Since $|J_0(\tilde \sigma)| \leq 1$, we get a bound of $r^{-1}$, which has the right scaling for the wave equation:
		$$
		|(\ref{temp1})| \leq \int_0^\infty \frac {\tilde \sigma}{(r^2+\tilde \sigma^2)^{3/2}} \dd \tilde \sigma = r^{-1}.
		$$
		
		The integral is also bounded by $r^{-3/2}$. Indeed, since $J_0(z) \les |z|^{-1/2}$,
		$$\begin{aligned}
		\int_0^\infty \frac {\tilde \sigma J_0(\tilde \sigma)}{(r^2+\tilde \sigma^2)^{3/2}} \dd \tilde \sigma &= \int_0^r \frac {\tilde \sigma J_0(\tilde \sigma)}{(r^2+\tilde \sigma^2)^{3/2}} \dd \tilde \sigma + \int_r^\infty \frac {\tilde \sigma J_0(\tilde \sigma)}{(r^2+\tilde \sigma^2)^{3/2}} \dd \tilde \sigma \\
		&\leq \int_0^r \frac {\tilde \sigma J_0(\tilde \sigma)}{r^3} \dd \tilde \sigma + \int_r^\infty \frac {\tilde \sigma J_0(\tilde \sigma)}{\tilde \sigma^3} \dd \tilde \sigma \les r^{-3/2}.
		\end{aligned}$$
		
		Following one more integration by parts, since $(z J_1(z))' = z J_0(z)$,
		$$\begin{aligned}
		\int_0^\infty \frac {\tilde \sigma J_0(\tilde \sigma)}{(r^2+\tilde \sigma^2)^{3/2}} \dd \tilde \sigma = 3 \int_0^\infty \frac {\tilde \sigma^2 J_1(\tilde \sigma)}{(r^2+\tilde \sigma^2)^{5/2}} \dd \tilde \sigma.
		\end{aligned}$$
		In the same way as above, we get a bound of $r^{-5/2}$.
		
		Iterating and using the fact that $(z^n J_n(z))' = z^n J_{n-1}(z)$, one gets that for every $n \geq 1$
		$$
		(\ref{temp1}) = (2n-1)!! \int_0^\infty \frac {\tilde \sigma^n J_{n-1}(\tilde \sigma)}{(r^2+\tilde \sigma^2)^{n+\frac 1 2}} \dd \tilde \sigma.
		$$
		This leads to arbitrary polynomial decay: due to (\ref{universal_bessel})
		$$\begin{aligned}
		(2n-1)!! \int_0^\infty \frac {\tilde \sigma^n J_{n-1}(\tilde \sigma)}{(r^2+\tilde \sigma^2)^{n+\frac 1 2}} \dd \tilde \sigma &\les \int_0^r \frac {\tilde \sigma^n J_{n-1}(\tilde \sigma)}{r^{2n+1}} \dd \tilde \sigma + \int_r^\infty \frac {\tilde \sigma J_{n-1}(\tilde \sigma)}{\tilde \sigma^{n+1}} \dd \tilde \sigma \\
		&\les (2n-1)!! \frac {r^{-n-\frac 1 2}}{n^{1/2}}.
		\end{aligned}$$
		
		In general, Stirling's formula asserts that $\Gamma(z) \sim \sqrt{2\pi} z^{z-\frac 1 2} e^{-z}$ as $\re z \to +\infty$. Then $(2n-1)!! = \frac {2^n}{\sqrt \pi} \Gamma(n+\frac 1 2) \sim 2^n n^{n+1} e^{-n}$. This leads to a bound of
		$$
		\frac {2^n n^{n+\frac 1 2} e^{-n}}{r^{n+\frac 1 2}}
		$$
		for the whole expression. Setting $n=\lfloor r \rfloor$, we end up with $(e/2)^{-r}$ for $r >> 1$, i.e.\;exponential decay.
		
		For an easier proof and a more accurate bound, consider the integral
		$$
		\int_{-\infty}^\infty \frac {H_1^+(\sigma)}{\sqrt{r^2+\sigma^2}} \dd \sigma = \int_{-\infty}^0 \frac {H_1^+(\sigma)}{\sqrt{r^2+\sigma^2}} \dd \sigma + \int_0^\infty \frac {H_1^+(\sigma)}{\sqrt{r^2+\sigma^2}} \dd \sigma.
		$$
		
		The integrand has a branching point at $\sigma=ir$ and a simple pole at $0$, as per (\ref{asymp}), which contributes half the value of its residue to the integral, since it is on the integration contour.
		
		We work with the branch of $\frac {H_1^+(\sigma)}{\sqrt{r^2+\sigma^2}}$ determined by a cut along $[ir, i\infty)$. By changing the contour of integration, the integral becomes
		$$\begin{aligned}
		&=\bigg(\int_{i\infty-0}^{ir-0} + \int_{ir+0}^{i\infty+0}\bigg) \frac {H_1^+(\sigma)}{\sqrt{r^2+\sigma^2}} \dd \sigma + \pi i \Res\bigg(\frac {H_1^+(\sigma)}{\sqrt{r^2+\sigma^2}}; \sigma=0\bigg) \\
		&= 2 \int_{ir+0}^{i\infty+0} \frac {H_1^+(\sigma)}{\sqrt{r^2+\sigma^2}} \dd \sigma + \frac 2 r.
		\end{aligned}$$
		However, a simple computation shows that
		$$
		H_1^+(-\sigma) = \ov{H_1^+}(\sigma) = H_1^-(\sigma)
		$$
		for $\sigma \in \R$, so the two integrals can be combined. Since $H_\alpha^\pm = J_\alpha \pm i Y_\alpha$ we get that
		$$
		\int_0^\infty \frac {J_1(\sigma)}{\sqrt{r^2+\sigma^2}} \dd \sigma - \frac 1 r = \int_{ir+0}^{i\infty+0} \frac {H_1^+(\sigma)}{\sqrt{r^2+\sigma^2}} \dd \sigma.
		$$
		For $r>>1$, this leads to a bound of
		\be\lb{imaginary}\begin{aligned}
			&\les \int_r^\infty \frac {e^{-\sigma}}{\sqrt \sigma \sqrt{\sigma^2-r^2}} \dd \sigma \les e^{-r} \int_r^\infty \frac {d\sigma}{\sqrt \sigma \sqrt{\sigma^2-r^2}} \les e^{-r} r^{-1/2},
		\end{aligned}\ee
		with the last bound due to scaling. Here we used the fact that $H_1^+(z) \sim \frac 1 {\sqrt z} e^{iz}$, for $|z|>>1$ in the upper half-plane $\im z > 0$.
		
		This Hankel function bound only works for large $r$. For $r<<1$, the initial $r^{-1}$ bound is the best possible.\\
		
		\noindent\textbf{II. Case $t>r$.} 
		Here we get the following $t^{-3/2}$ bound:
		$$\begin{aligned}
		\int_t^\infty \frac {J_1(\sqrt{\tau^2-r^2})} {\sqrt{\tau^2-r^2}} \dd \tau &= \int_{t-r}^\infty \frac {J_1(\sqrt{\sigma^2+2r\sigma})} {\sqrt{\sigma^2+2r\sigma}} \dd \sigma \\
		&= \int_{\sqrt{t^2-r^2}}^\infty \frac {J_1(\tilde \sigma)} {r+\sigma(\tilde \sigma)} \dd \tilde \sigma \\
		&= \int_{\sqrt{t^2-r^2}}^\infty \bigg(\int_{\tilde \sigma}^\infty \frac {\sigma'(\sigma^*) \dd \sigma^*} {(r+\sigma(\sigma^*))^2}\bigg) J_1(\tilde \sigma) \dd \tilde \sigma \\
		&= \int_{\sqrt{t^2-r^2}}^\infty \frac {\sigma'(\tilde \sigma)}{(r+\sigma(\tilde \sigma))^2} \bigg(\int_{\sqrt{t^2-r^2}}^{\tilde \sigma} J_1(\sigma^*) \dd \sigma^*\bigg) \dd \tilde \sigma \\
		&\leq \frac 1 t \sup_{\tilde \sigma \geq \sqrt{t^2-r^2}} |J_0(\sqrt{t^2-r^2}) - J_0(\tilde \sigma)| \\
		&\les \frac 1 {t \langle t^2-r^2\rangle^{1/4}}.
		\end{aligned}$$
		Replacing the integral up to $\tilde \sigma$ with the one up to $+\infty$ in the above, we get the main term in the asymptotic expansion:
		$$
		\chi_{t>r} \frac {J_0(\sqrt{t^2-r^2})} {t}.
		$$
		The remainder in this expansion is
		\be\lb{temp2}\begin{aligned}
			\int_{\sqrt{t^2-r^2}}^\infty \frac {\sigma'(\tilde \sigma)}{(r+\sigma(\tilde \sigma))^2} \bigg(\int_{\tilde \sigma}^\infty J_1(\sigma^*) \dd \sigma^*\bigg) \dd \tilde \sigma &= \int_{\sqrt{t^2-r^2}}^\infty \frac {\sigma'(\tilde \sigma) J_0(\tilde \sigma)}{(r+\sigma(\tilde \sigma))^2} \dd \tilde \sigma \\
			&= \int_{\sqrt{t^2-r^2}}^\infty \frac {J_0(\tilde \sigma) \tilde \sigma \dd \tilde \sigma}{(r^2+\tilde \sigma^2)^{3/2}}.
		\end{aligned}
		\ee
		Since $J_0$ is bounded, a rough estimation gives a $t^{-1}$ bound. On the other hand, integration by parts gives
		$$\begin{aligned}
		&= \frac {\sqrt{t^2-r^2}J_1(\sqrt{t^2-r^2})}{t^3} + 3 \int_{\sqrt{t^2-r^2}}^\infty \frac {\tilde \sigma^2 J_1(\tilde \sigma)}{(r^2+\tilde \sigma^2)^{5/2}} \\
		&\les \frac {(t^2-r^2)^{1/4}} {t^3} + \min(r^{-5/2}, (t^2-r^2)^{-5/4}).
		\end{aligned}$$
		Note that $|J_0(z)| \les \langle z \rangle^{-1/2}$. Also, for both $t \sim r$ and $t>>r$,
		$$
		t \les \max(\sqrt{t^2-r^2}, r) \les t.
		$$
		Thus, when $t>r$, we get a bound of
		$$
		(\ref{temp2}) \les \frac 1 {t \langle t^2-r^2 \rangle^{1/4}}+t^{-5/2}.
		$$
		
		This is the desired bound when $r>1$ or $r \sim 1$.
		
		For $r<1$ and $t>r$, this bound can be improved. The integral
		$$
		\int_{\tilde t}^\infty \frac {J_1(\sqrt{\sigma^2+2r\sigma})} {\sqrt{\sigma^2+2r\sigma}} \dd \sigma
		$$
		is uniformly bounded for $\tilde t := t-r \geq 0$ and $r=0$ and so is its $r$ derivative:
		$$
		\int_{\tilde t}^\infty \frac {J_1'(\sqrt{\sigma^2+2r\sigma}) \sigma} {\sigma^2+2r\sigma} - \frac {J_1(\sqrt{\sigma^2+2r\sigma}) \sigma} {(\sigma^2+2r\sigma)^{3/2}} \dd \sigma = - \int_{\tilde t}^\infty \frac {J_2(\sqrt{\sigma^2+2r\sigma})}{\sigma^2+2r\sigma} \dd \sigma.
		$$
		This quantity is uniformly bounded for $r<1$ and $t>r$, which establishes the desired estimate for small $t$.
	\end{proof}
	
	One goal of the estimates in Proposition \ref{cosine_kernel} is to bound the cosine kernel $C_1$ in $L^p_t$, $1 \leq p \leq \infty$, in particular in $L^1_t$ and $L^\infty_t$.
	
	\begin{lemma}\lb{c1bounds} The cosine kernel $C_1$ satisfies the following bounds:
		$$
		\|C_1(r, t)\|_{L^1_t} \les \langle r \rangle^{-1/2},\ \|C_1(r, t)\|_{L^\infty_t} \les r^{-1}.
		$$
		In addition, for $p<4$
		$$
		\|C_1(r, t)\|_{L^p_t} \les r^{1/p-3/2}
		$$
		and for $p>4$
		$$
		\|C_1(r, t)\|_{L^p_t} \les r^{-1-1/p}.
		$$
		For $p=4$, the bound holds in the weak-type Lebesgue space $L^{4, \infty}$:
		$$
		\|C_1\|_{L^{4, \infty}_t} \les r^{-5/4}.
		$$
	\end{lemma}
	\begin{observation} In order to compare these estimates, let us count the powers of decay.
		
		For $p \in [1, 4]$, our estimates give $3/2$ powers of decay, while for $p \in [4, \infty]$ they give $1+2/p$ powers. In addition, for the whole range $p \in [1, \infty]$ we get estimates with exactly one power of decay, like for the wave equation.
		
		By comparison, for the wave equation propagators, all estimates have equally many powers of decay: $\frac {\sin(t\sqrt{-\Delta})}{\sqrt{-\Delta}}$ has two powers of decay, $\frac {\cos(t\sqrt{-\Delta})}{-\Delta}$ has one power, and in general $\frac {e^{it\sqrt{-\Delta}}}{(-\Delta)^s}$ has $3-s$ powers of decay. 
	\end{observation}
	
	\begin{proof} The $L^\infty_t$ norm of $C_1(r, t)$ is at least $1/r$:
		$$
		\|C_1(r, t)\|_{L^\infty_t} \ges 1/r,
		$$
		due to the jump discontinuity of this size near the light cone. This is also an upper bound, as shown by the estimates of Proposition \ref{cosine_kernel}.
		
		However, for $r>1$, one can do better than $1/r$ in almost every situation.
		
		When $t>(1+\epsilon)r$ (so away from the light cone) the kernel is bounded by $t^{-3/2} \leq r^{-3/2}$. For $0 \leq t < r$ there is an exponential bound $e^{-r} r^{-1/2} \les r^{-3/2}$. Even for $t>r+\epsilon$ there is a bound of $r^{-5/4}$.
		
		However, when $t \sim r$ and $t>r$, $r^{-1}$ is still the best possible bound.
		
		Thus, further improvements are possible, but not in the $L^\infty$ norm. For $p<4$ and $r>1$,
		
		For $p>4$, $t^{-p}(t^2-r^2)^{-p/4}$ is no longer integrable, so we need to use a different bound: on the interval $[r, r+\rho]$
		$$
		\|1/r\|_{L^p_t([r, r+\rho])} \les r^{-1} \rho^{1/p},
		$$
		while on the interval $[r+\rho, 2r]$
		$$
		\bigg\|\frac 1 {r^{5/4}(t-r)^{1/4}}\bigg\|_{L^p_t} \les r^{-5/4} \rho^{1/p-1/4}
		$$
		and on the interval $[2r, \infty)$ there is a $r^{1/p-3/2}$ bound. Setting $\rho=r^{-1}$, we get an overall $r^{-1-1/p}$ bound.
		
		We next estimate the $L^1_t$ norm. For $r<1/2$, the $[0, r]$ interval contributes a term of size $1$, while the $[r, \infty)$ interval also contributes
		$$
		\int_r^1 dt + \int_1^\infty t^{-3/2} dt \les 1.
		$$
		
		For large $r$, the $[0, r]$ interval has an exponentially small contribution, while on its complement we get, by scaling,
		$$
		\int_r^\infty \frac {dt}{t(t^2-r^2)^{1/4}} \les r^{-1/2}.
		$$
		
		For $p=4$, in order to get a bound without logarithms, we need to use a weak-type norm instead:
		$$
		\|C_1\|_{L^{4, \infty}_t} \les r^{-5/4}.
		$$
		Indeed, by rescaling
		$$
		\bigg\|\chi_{t>r} \frac 1 {t(t^2-r^2)^{1/4}}\bigg\|_{L^{4, \infty}_t} = r^{-5/4} \bigg\|\chi_{t>1} \frac 1 {t(t^2-1)^{1/4}}\bigg\|_{L^{4, \infty}_t} \les r^{-5/4}.
		$$
	\end{proof}
	
	\begin{observation} Using properties of convolution on Lorentz spaces,
		$$
		L^{p, 2} \ast L^{q, \infty} \mapsto L^{r, 2},\ \frac 1 p + \frac 1 q - 1 = \frac 1 r.
		$$
	\end{observation}
	
	This implies the following reversed Strichartz estimates for the Klein--Gordon equation.
	\begin{theorem}\lb{thm_partial_results} The operators $E_{1/2}$, $S_{1/2}$, and $C_{1/2}$ are bounded between the following spaces:
		$$
		S_{1/2}, C_{1/2}, E_{1/2}: L^2 \to L^\infty_x L^2_t \cap L^{12, 2}_x L^2_t \cap L^{6, 2}_x L^\infty_t \cap L^{24/5, 2}_x L^{8, 2}_t.
		$$
	\end{theorem}
	The first and third estimates are also true for the free wave equation, while the second and fourth are better by a quarter power of decay --- meaning that they manifest $3/4$ powers of decay in total.
	\begin{proof} Due to Lemma \ref{c1bounds}, $C_1$ is bounded between the following pairs of dual spaces:
		$$
		\|C_1(r, t)\|_{L^1_t} \les 1 \implies C_1: L^1_x L^2_t \to L^\infty_x L^2_t,
		$$
		$$
		\|C_1(r, t)\|_{L^1_t} \les r^{-1/2} \implies C_1: L^{12/11, 2}_x L^2_t \to L^{12, 2}_x L^2_t,
		$$
		and
		$$
		\|C_1(r, t)\|_{L^\infty_t} \les r^{-1} \implies C_1: L^{6/5, 2}_x L^1_t \to L^{6, 2}_x L^\infty_t.
		$$
		
		In the first and third inequalities, the $C_1$ integral kernel has exactly one power of decay, like the free wave equation, while in the second it has $3/2$ powers of decay, something specific to Klein--Gordon.
		
		More generally, when estimated in $L^p_t$, the $C_1$ integral kernel has $3/2$ powers of decay for $p<4$ and $1+2/p$ for $p>4$, with $p=4$ being a special endpoint case:
		$$
		\|C_1(r, t)\|_{L^{4, \infty}_t} \les r^{-5/4} \implies C_1:L^{24/19, 2}_x L^{8/7, 2}_t \to L^{24/5, 2}_x L^{8, 2}_t.
		$$
		Indeed, $L^{8/7, 2} \ast L^{4, \infty} \mapsto L^{8, 2}$.
		
		In order to obtain estimates for $S_{1/2}$ and $C_{1/2}$, we use the $T T^*$ method:
		$$
		\frac{\sin(t\sqrt{-\Delta+1})}{\sqrt{-\Delta+1}} \frac{\sin(s\sqrt{-\Delta+1})}{\sqrt{-\Delta+1}} = \frac {\cos((t-s)\sqrt{-\Delta+1})}{-\Delta+1} - \frac {\cos((t+s)\sqrt{-\Delta+1})}{-\Delta+1}
		$$
		and
		$$
		\frac{\cos(t\sqrt{-\Delta+1})}{\sqrt{-\Delta+1}} \frac{\cos(s\sqrt{-\Delta+1})}{\sqrt{-\Delta+1}} = \frac {\cos((t-s)\sqrt{-\Delta+1})}{-\Delta+1} + \frac {\cos((t+s)\sqrt{-\Delta+1})}{-\Delta+1}.
		$$
		Thus, in order to prove that $S_{1/2}, C_{1/2} \in B(L^2, L^p_x L^q_t)$, it suffices to prove that $C_1$ acting by convolution in $t$ is bounded from $L^{p'}_x L^{q'}_t$ to $L^p_x L^q_t$, where $p'$ and $q'$ are the dual exponents of $p$ and $q$.
	\end{proof}
	
	\begin{observation} Interpolating between the second and the fourth estimate we also obtain a homogeneous spacetime norm, namely $L^{16/3}_{x, t}$.
		
		Numerology: using weights of $1/6$ for the second estimate ($L^{12, 2}_x L^2_t$) and $5/6$ for the fourth estimate ($L^{24/5, 2}_x L^{8, 2}_t$), we obtain $L^{16/3, 2}$ in both space and time:
		$$
		\frac 3 {16} = \frac 1 6 \cdot \frac 1 {12} + \frac 5 {6} \cdot \frac 5 {24}
		$$
		and
		$$
		\frac 3 {16} = \frac 1 6 \cdot \frac 1 2 + \frac 5 6 \cdot \frac 1 8.
		$$
		Starting from the other two estimates, we obtain that the solution is in $L^8_{x, t}$, which is also valid for the wave equation.
	\end{observation}
	
	\subsection{Fractional integration}
	Next, we perform fractional integration of order $\alpha$ in $t$, in order to estimate
	$$
	D^{-\alpha}_t \frac {e^{it\sqrt{-\Delta+1}}}{\sqrt{-\Delta+1}} = \frac {e^{it\sqrt{-\Delta+1}}}{(-\Delta+1)^{(1+\alpha)/2}}=E_{\frac{\alpha+1}{2}}(t),
	$$
	starting from the known expression of $S_{1/2}$.
	
	Fractional integration $D^{-\alpha}_t$ is given by convolution with $|t|^{\alpha-1}$ with or without some complex phase, see below. 
	
	This spends $\alpha$ powers of decay, hence this is a bounded operator from $L^p_t$ to $L^q_t$, $1<p, q<\infty$, $1/p-1/q=\alpha$, by Young's inequality, as well as from $L^1$ to $L^{1/(1-\alpha), \infty}$ and from $L^{1/\alpha, 1}$ to $L^\infty$ at the endpoints.
	
	For $\Re \alpha \in (0, 1)$, there are two linearly independent choices of fractional integration operators, which differ roughly by a Hilbert transform.
	
	We first perform some auxiliary computations to determine which one to use.
	
	\begin{lemma} For $0 < \Re \alpha < 1$ and $f \in \S$,
		\be\lb{abst}
		\mc F \bigg[f \ast \frac {|t|^{\alpha-1}}{2 \Gamma(\alpha)}\bigg] = \frac {\cos(\frac \pi 2 \alpha)}{|\xi|^\alpha} \widehat f(\xi)
		\ee
		and
		\be\lb{ant}
		\mc F \bigg[f \ast \frac {t^{\alpha-1}}{2 i^{\alpha-1} \Gamma(\alpha)} \bigg] = \chi_{\xi>0} \frac {\sin(\pi \alpha)}{|\xi|^\alpha} \widehat f(\xi),
		\ee
		where $t^{\alpha-1}$ is defined so as to be analytic in the upper half-plane, see (\ref{ant'}).
		
		Same holds for $0 < \Re \alpha < 2$, $\alpha \ne 1$, when $f \in \S$ has $\int_{-\infty}^\infty f(t) \dd t = 0$ and more generally for $0 < \Re \alpha < n$, $\alpha \not \in \Z$, when the function $f$'s first $n-1$ moments vanish.
	\end{lemma}
	
	In general, both the expression in (\ref{abst}) and the expression in (\ref{ant}) have singularities at $\alpha=1$ and more generally at $\alpha \in \Z$, but we can ignore them in the case of functions with sufficiently many vanishing moments.
	
	For example, at $\alpha=1$, the right-hand side in (\ref{abst}) and in (\ref{ant}) vanishes when $f$'s first moment vanishes, as appropriate.
	
	\begin{proof}
		For any Schwartz-class function $f \in \S$ and $0 < \Re \alpha < 1$,
		$$\begin{aligned}
		f \ast |t|^{\alpha - 1} &= \int_{\R} f(t-\tau) |\tau|^{\alpha-1} \dd \tau = \int_\R \int_\R \widehat f(\xi) e^{i(t-\tau) \xi} |\tau|^{\alpha-1} \dd \xi \dd \tau \\
		&= \int_\R \widehat f(\xi) \bigg(\lim_{R \to \infty} \int_{-R}^R e^{i(t-\tau)\xi} |\tau|^{\alpha-1} \dd \tau \bigg) \dd \xi
		\end{aligned}$$
		and same for the other fractional integration operators we consider.
		
		Let $0 < \Re \alpha < 1$. For $\xi > 0$, convolution with $\chi_{t \leq 0} |t|^{\alpha-1}$ has the following effect on $e^{it\xi}$:
		\be\lb{rel1}
		e^{it\xi} \ast \chi_{t \leq 0} |t|^{\alpha-1} = \int_{-\infty}^0 e^{i(t-\tau)\xi} |\tau|^{\alpha-1} \dd \tau = e^{it\xi} \int_0^\infty e^{i\tau\xi} \tau^{\alpha-1} \dd \tau = \frac{e^{it\xi} i^\alpha \Gamma(\alpha)}{\xi^\alpha},
		\ee
		by a change of contour.
		
		Likewise, again for $\xi>0$
		\be\lb{rel2}
		e^{it\xi} \ast \chi_{t \geq 0} t^{\alpha-1} = \int_0^\infty e^{i(t-\tau)\xi} \tau^{\alpha-1} \dd \tau = e^{it\xi} \int_0^\infty e^{-i\tau\xi} \tau^{\alpha-1} \dd \tau = \frac{e^{it\xi} (-i)^\alpha \Gamma(\alpha)}{\xi^\alpha}.
		\ee
		Powers of $\pm i$ are defined here using the main branch of the logarithm, $(\pm i)^\alpha = e^{\pm i \frac \pi 2 \alpha}$.
		
		Hence these fractional integration operators behave in the same way for all Schwartz-class functions and, by extension, for all spaces of functions we are interested in (in which Schwartz-class functions are dense).
		
		For $\xi<0$, one similarly gets:
		\be\lb{rel3}\begin{aligned}
			e^{it\xi} \ast \chi_{t \geq 0} t^{\alpha-1} = \int_0^\infty e^{i(t-\tau)\xi} \tau^{\alpha-1} \dd \tau = e^{it\xi} \int_0^\infty e^{-i\tau\xi} \tau^{\alpha-1} \dd \tau = \frac{e^{it\xi} i^\alpha \Gamma(\alpha)}{|\xi|^\alpha}
		\end{aligned}\ee
		and
		\be\lb{rel4}
		e^{it\xi} \ast \chi_{t \leq 0} |t|^{\alpha-1} = \int_{-\infty}^0 e^{i(t-\tau)\xi} |\tau|^{\alpha-1} \dd \tau = e^{it\xi} \int_0^\infty e^{i\tau\xi} \tau^{\alpha-1} \dd \tau = \frac{e^{it\xi} (-i)^\alpha \Gamma(\alpha)}{|\xi|^\alpha}.
		\ee
		This is the opposite of the previous case $\xi>0$.
		
		Hence, for $0 < \Re \alpha < 1$ convolution with $|t|^{\alpha-1}$ acts as follows, for all $\xi \in \R$:
		$$
		e^{it\xi} \ast |t|^{\alpha-1} = e^{it\xi} \ast \chi_{t \leq 0} |t|^{\alpha-1} + e^{it\xi} \ast \chi_{t \geq 0} t^{\alpha-1} = \frac {2 e^{it\xi} \cos(\frac \pi 2 \alpha) \Gamma(\alpha)}{|\xi|^\alpha}.
		$$
		
		Note that convolution with $\frac 1 {\Gamma(\alpha)} |t|^{\alpha-1}$ has an analytic extension to $\alpha=0$, whereas convolution with the family $|t|^{\alpha-1}$ has a pole at $\alpha=0$.
		
		On the other hand, convolution with
		\be\lb{ant'}
		t^{\alpha-1} := \chi_{t \leq 0} e^{i\pi(\alpha-1)} |t|^{\alpha-1} + \chi_{t \geq 0} t^{\alpha-1},
		\ee
		which has an analytic extension to the upper half-plane as defined, acts as follows:
		$$
		e^{it\xi} \ast t^{\alpha-1} = \left\{\begin{aligned}
		&\frac {2 e^{it\xi} i^{\alpha-1} \sin(\pi \alpha) \Gamma(\alpha)}{|\xi|^\alpha},&&\xi > 0,\\
		&0,&&\xi < 0.
		\end{aligned}\right.
		$$
		Indeed, for $\xi>0$
		$$\begin{aligned}
		e^{i\pi(\alpha-1)} (e^{it\xi} \ast \chi_{t \leq 0} |t|^{\alpha-1}) + e^{it\xi} \ast \chi_{t \geq 0} t^{\alpha-1} &= \frac {e^{it\xi} \Gamma(\alpha)}{\xi^\alpha} [e^{i\pi(\alpha-1)} i^\alpha + (-i)^\alpha]
		\\
		&= \frac {e^{it\xi} \Gamma(\alpha)}{\xi^\alpha} [-e^{i\frac {3\pi} 2 \alpha} + e^{-i\frac {\pi} 2 \alpha}] \\
		&= \frac {e^{it\xi} \Gamma(\alpha)}{\xi^\alpha} e^{i\frac {\pi} 2 \alpha} \cdot 2i \sin(-\pi \alpha).
		\end{aligned}
		$$
		For $\xi<0$
		$$\begin{aligned}
		e^{i\pi(\alpha-1)} (e^{it\xi} \ast \chi_{t \leq 0} |t|^{\alpha-1}) + e^{it\xi} \ast \chi_{t \geq 0} t^{\alpha-1} &= \frac {e^{it\xi} \Gamma(\alpha)}{\xi^\alpha} [e^{i\pi(\alpha-1)} (-i)^\alpha + i^\alpha]
		\\
		&= \frac {e^{it\xi} \Gamma(\alpha)}{\xi^\alpha} [-e^{i\frac {\pi} 2 \alpha} + e^{i\frac {\pi} 2 \alpha}]  = 0.
		\end{aligned}
		$$
		
		The conjugate of this function has an analytic extension to the lower half-plane in $\alpha$. Overall, this is not an analytic function of $\alpha$, due to the $i^{\alpha-1}$ factor.
		
		One can also summarize relations (\ref{rel1}-\ref{rel4}) as
		$$
		e^{it\xi} \ast \chi_{t \geq 0} t^{\alpha-1} = \frac {e^{it\xi} \Gamma(\alpha)}{(i\xi)^\alpha}
		$$
		and
		$$
		e^{it\xi} \ast \chi_{t \leq 0} t^{\alpha-1} = \frac {e^{it\xi} i^\alpha \Gamma(\alpha)}{\xi^\alpha},
		$$
		for $0 < \Re \alpha < 1$, where the powers are defined using the main branch of the complex logarithm.
		
		The Fourier transform of $\chi_{t \geq 0}$ is  
		$$
		\frac 1 {t+i0} = \delta_0 - i\;p.v. \frac 1 t.
		$$
		
		So convolution with $\chi_{t \geq 0}$ is the same as the Fourier multiplier $\frac 1 {i\xi}$ for Schwartz functions whose integral is zero.
		
		Then for such functions
		$$
		\mc F(f \ast \chi_{t \geq 0}) = \frac 1 {i\xi} \widehat f(\xi),\ \mc F(f \ast \chi_{t \leq 0}) = \frac i {\xi} \widehat f(\xi).
		$$
		
		So at least for functions whose integral (first moment) is zero we can extend relations (\ref{rel1}-\ref{rel4}), as well as our conclusions, to $0 < \Re \alpha < 2$.
	\end{proof}
	
	
	Thus, starting from $S_{1/2}$, we can compute by fractional integration all the other kernels $S_{(1+\alpha)/2}$, $C_{(1+\alpha)/2}$, and $E_{(1+\alpha)/2}$ for $0 < \Re \alpha < 1$.
	
	We are also interested in the $\Re\alpha=0$ and $\Re\alpha=1$ cases, as well as in more than one order of fractional integration, $\Re \alpha > 1$. Due to analytic continuation, the formulas above remain valid in these cases, when properly interpreted.
	
	
	\begin{lemma}\lb{combinations} For $\Re\alpha>0$, $\alpha \not \in \Z$,
		$$
		S_{1/2} \ast |t|^{\alpha-1} = 2 \cos(\frac \pi 2 \alpha) \Gamma(\alpha) S_{(1+\alpha)/2},
		$$
		$$
		S_{1/2} \ast t^{\alpha-1} = -i^\alpha \sin(\pi\alpha) \Gamma(\alpha) E_{(1+\alpha)/2},
		$$
		$$
		S_{1/2} \ast \Re \bigg[\frac{t^{\alpha-1}}{i^\alpha \sin(\pi\alpha)  \Gamma(\alpha)}\bigg] = - C_{(1+\alpha)/2}.
		$$
		Here $t^{\alpha-1}$ is defined using the main branch of the complex logarithm.
	\end{lemma}
	\begin{observation}
		As $\alpha$ approaches $0$, the last two expressions have meaningful limits: as
		$$
		\lim_{\alpha \to 0} \sin(\pi \alpha) \Gamma(\alpha) = \pi,
		$$
		we get
		$$
		S_{1/2} \ast \frac 1 {\pi (t+i0)} = -E_{1/2},\ S_{1/2} \ast p.v. \frac 1 {\pi t} = -C_{1/2}.
		$$
		The former, $\frac 1 {\pi (t+i0)}$, corresponds to the Cauchy kernel for the upper half-plane; the latter, $p.v. \frac 1 {\pi t}$, is the Hilbert kernel.
	\end{observation}
	
	
	\begin{proof}
		For each fixed $r$, the kernel of $S_{1/2}$ is the sum of a bounded integrable function, which decays like $t^{-3/2}$ at infinity (the first term in the asymptotic expansion behaves roughly like $t^{-3/2} \sin t$), and a signed measure.
		
		Also, its $t$ integral is zero, because it is an odd measure. Hence, we can apply the previous result for $0 < \Re \alpha < 2$.
		
		But in fact all its moments are zero, because its Fourier support is outside the interval $(-1, 1)$. Indeed, its Fourier transform as a function of $t$ is
		$$
		[\mc F_t {S_{1/2}}](\tau) = \sgn(\tau) \chi_{|\tau| \geq 1} \frac 1 {4\pi} \frac {\sin(\sqrt{\tau^2-1}|x|)}{|x|}.
		$$
		
		So $\Re \alpha > 0$ can be arbitrarily large.
		
		The key fact is that $t^{1-\alpha}$ and $|t|^{1-\alpha}$ are linearly independent for $\alpha \not \in 2\Z+1$. Thus, choosing suitable combinations of $t^{1-\alpha}$ and $|t|^{1-\alpha}$, we can retrieve all the integral kernels in which we are interested.
		
		Due to (\ref{abst}), for any $A > 0$
		$$
		\frac {\sin(tA)}{A} \ast |t|^{\alpha-1} = 2 \cos(\frac \pi 2 \alpha) \Gamma(\alpha) \frac {\sin(tA)}{A^{\alpha+1}}.
		$$
		Therefore
		$$
		S_{1/2}(t) \ast |t|^{\alpha-1} = 2 \cos\Big(\frac \pi 2 \alpha\Big) \Gamma(\alpha) S_{(1+\alpha)/2}(t).
		$$
		
		By (\ref{ant}), for any $A > 0$
		$$
		\frac {\sin(tA)}{A} \ast t^{\alpha-1} = 2i^{\alpha-1} \sin(\pi \alpha) \Gamma(\alpha) \frac {e^{itA}}{2i A^{1+\alpha}}.
		$$
		Hence
		$$
		S_{1/2}(t) \ast t^{\alpha-1} = -i^{\alpha} \sin(\pi \alpha) \Gamma(\alpha) E_{(1+\alpha)/2}(t)
		$$
		or
		$$
		S_{1/2}(t) \ast t^{\alpha-1} = -i^{\alpha} \sin(\pi \alpha) \Gamma(\alpha) E_{(1+\alpha)/2}(t)
		$$
		This has a removable singularity at $\alpha=1$.
		
		Thus
		$$
		S_{1/2}(t) \ast \Re \frac{t^{\alpha-1}}{i^{\alpha} \sin(\pi\alpha) \Gamma(\alpha)} = -C_{(1+\alpha)/2}(t).
		$$
	\end{proof}
	
	By this sort of fractional integration, we obtain the following estimates:
	
	\begin{proposition}\lb{fractional_powers} Let $\alpha = a+ib$, $a, b \in \R$, and $0 < a=\re \alpha < 5/2$. Then the integral kernel of $E_{(1+\alpha)/2}$, hence also those of $S_{(1+\alpha)/2}$ and $C_{(1+\alpha)/2}$, satisfies the following bounds.
		
		For $0 \leq t < r$,
		$$
		\big|E_{\frac{\alpha+1}{2}}\big|(r, t) = \bigg|\frac {e^{it\sqrt{-\Delta+1}}}{(-\Delta+1)^{(1+\alpha)/2}}\bigg|(r, t) \les \frac {e^{-r}}{t^{1-a} \sqrt r}
		$$
		when $r\geq 1$ is large and
		$$
		\big|E_{\frac{\alpha+1}{2}}\big|(r, t) =\bigg|\frac {e^{it\sqrt{-\Delta+1}}}{(-\Delta+1)^{(1+\alpha)/2}}\bigg|(r, t) \les \frac 1 {r(r-t)^{1-a}}
		$$
		when $r\leq 1$ is small.
		
		For $t>r$,
		\be\lb{complicated}
		\big|E_{\frac{\alpha+1}{2}}\big|(r, t) =\bigg|\frac {e^{it\sqrt{-\Delta+1}}}{(-\Delta+1)^{(1+\alpha)/2}}\bigg|(r, t) \les 
		\frac 1 {t^a \langle t^2-r^2 \rangle^{3/4-a/2}}
		\ee
		when $r$ is large. This bound is true and, in addition, the integral kernel is also uniformly bounded when $r$ is small.
	\end{proposition}
	These bounds are true for all combinations --- sine, cosine, and complex exponentials --- but for some particular combinations one can do strictly better.
	
	\begin{proof} We perform the computation for $t>0$ only. In light of Lemma \ref{combinations}, it suffices to estimate $S_{1/2} \ast \chi_{t\geq 0} t^{\alpha-1}$ and $S_{1/2} \ast \chi_{t\leq 0} |t|^{\alpha-1}$, meaning the integrals
		\be\lb{first}
		S_{1/2} \ast \chi_{t\geq 0} = \int_t^\infty \frac {J_1(\sqrt{\tau^2-r^2})}{\sqrt{\tau^2-r^2}} (\tau-t)^{\alpha-1} \dd \tau
		\ee
		when $t > r$ and
		\be\lb{first'}
		S_{1/2} \ast \chi_{t\geq 0} = \int_r^\infty \frac {J_1(\sqrt{\tau^2-r^2})}{\sqrt{\tau^2-r^2}} (\tau-t)^{\alpha-1} \dd \tau - \int_\R \frac 1 r \delta_r(\tau)(\tau-t)^{\alpha-1} \dd \tau
		\ee
		when $0 \leq t < r$, as well as
		\be\lb{second}
		S_{1/2} \ast \chi_{t\leq 0} |t|^{\alpha-1} = \int_{-\infty}^{-r} \frac {J_1(\sqrt{\tau^2-r^2})}{\sqrt{\tau^2-r^2}} |\tau-t|^{\alpha-1} \dd \tau = \int_r^{\infty} \frac {J_1(\sqrt{\tau^2-r^2})}{\sqrt{\tau^2-r^2}} (\tau+t)^{\alpha-1} \dd \tau
		\ee
		for the other kernel. Since we assumed $t>0$, it follows that $t>-r$, so here we only have one case.
		
		We first evaluate (\ref{first}) and (\ref{first'}). The computation is similar to the one in Proposition \ref{cosine_kernel}. Let $\tau=\sigma+r$. We again make the change of variable
		$$
		\tilde \sigma(\sigma) = \sqrt{\sigma^2+2r\sigma} = \sqrt{\tau^2-r^2},
		$$
		with the inverse (\ref{inverse}). One also has to consider the Jacobian factor, as in (\ref{inverse'}).
		
		\noindent\textbf{I. Case $0 \leq t < r$.} On this interval, which corresponds to the exterior of the light cone, we need to compute expression (\ref{first'}). Start with
		\be\lb{first_term}\begin{aligned}
			&\int_r^\infty \frac {J_1(\sqrt{\tau^2-r^2})} {\sqrt{\tau^2-r^2}} (\tau-t)^{\alpha-1} \dd \tau \\
			&= \int_0^\infty \frac {J_1(\sqrt{\sigma^2+2r\sigma})} {\sqrt{\sigma^2+2r\sigma}} (\sigma+r-t)^{\alpha-1} \dd \sigma \\
			&= \int_0^\infty \frac {J_1(\tilde \sigma)} {(r+\sigma(\tilde \sigma))(r+\sigma(\tilde \sigma)-t)^{1-\alpha}} \dd \tilde \sigma \\
			&= \int_0^\infty \frac \partial {\partial \tilde \sigma} \bigg(\frac {-1} {(r+\sigma(\tilde \sigma))(r+\sigma(\tilde \sigma)-t)^{1-\alpha}}\bigg) (J_0(0)-J_0(\tilde \sigma)) \dd \tilde \sigma \\
			&\leq \frac 1 {r(r-t)^{1-a}} \sup_{\tilde \sigma \geq 0} |J_0(0)-J_0(\tilde \sigma)| \leq \frac 2 {r(r-t)^{1-a}}.
		\end{aligned}\ee
		
		For the boundary term in this integration by parts to vanish as $\tilde \sigma \to \infty$, one needs that $-1/2-(2-a) < 0$ (where half a power of decay comes from the Bessel function), so $a=\Re\alpha<5/2$. One needs instead to use repeated integration when $a \geq 5/2$.
		
		The contribution of the free wave equation term is
		\be\lb{wave_eqn_term}
		\int \frac 1 r \delta_r(\tau) (\tau-t)^{\alpha-1} \dd \tau = \frac 1 {r(r-t)^{1-\alpha}}.
		\ee
		
		This suggests a cancellation. Same as in (\ref{imaginary}), we get polynomial and even exponential decay, as for large $r$ the difference has size
		$$
		(\ref{first'}) \les \int_r^\infty \frac {e^{-\sigma} \dd \sigma} {\sqrt \sigma \sqrt{r^2-\sigma^2} (\sqrt {r^2-\sigma^2} - t)^{1-\alpha}} \dd \sigma.
		$$
		Note that $\sqrt {r^2-\sigma^2}$ is purely imaginary on this domain, because $\sigma>r$, so
		$$
		|\sqrt {r^2-\sigma^2} - t| = \sqrt{\sigma^2-r^2 + t^2}.
		$$
		We divide the domain $[r, \infty)$ into two parts, $\sigma \in [r, \sqrt{r^2+t^2}]$ and $\sigma \geq \sqrt{r^2+t^2}$, leading to a bound of
		$$
		(\ref{first'}) \les \frac {e^{-r}} {t^{1-a} \sqrt r} + \frac {e^{-r}} {r^{3/2-a}} \les \frac {e^{-r}} {t^{1-a} \sqrt r}.
		$$
		
		For $r<1$, (\ref{first_term}) is uniformly bounded, so the size of the expression (\ref{first'}) is determined by (\ref{wave_eqn_term}).\\
		
		\noindent\textbf{II. Case $t>r$.} Here we compute (\ref{first}), which is independent of $r$.
		
		Again, for $r<1$ the expression is uniformly bounded. For $r>1$, we split the domain of integration. Fix $T>t$, to be specified below. Integrating by parts on $[T, \infty)$ we get
		$$\begin{aligned}
		\int_T^\infty \frac {J_1(\sqrt{\tau^2-r^2})} {\sqrt{\tau^2-r^2}} (\tau-t)^{\alpha-1} \dd \tau &= \int_{T-r}^\infty \frac {J_1(\sqrt{\sigma^2+2r\sigma})} {\sqrt{\sigma^2+2r\sigma}} (\sigma+r-t)^{\alpha-1} \dd \sigma \\
		&= \int_{\sqrt{T^2-r^2}}^\infty \frac {J_1(\tilde \sigma)} {(r+\sigma(\tilde \sigma))(r+\sigma(\tilde \sigma)-t)^{1-\alpha}} \dd \tilde \sigma \\
		&\les \frac 1 {T (T - t)^{1-a}\langle T^2-r^2\rangle^{1/4}} \\
		&\les \frac 1 {t (T - t)^{1-a}\langle t^2-r^2\rangle^{1/4}}.
		\end{aligned}$$
		The boundary term as $\tilde \sigma \to \infty$ in the integration by parts only vanishes if $-1/2-(2-a)<0$, so if $a=\re \alpha<5/2$.
		
		On the remaining interval $[t, T]$, since
		$$
		\bigg|\frac {J_1(z)}{z}\bigg| \les \langle z \rangle^{-3/2},
		$$
		we obtain an estimate of
		$$
		\frac {(T-t)^a}{\langle t^2-r^2 \rangle^{3/4}}.
		$$
		Setting $T = t + \langle t^2-r^2 \rangle^{1/2}/t$ leads to an overall bound of
		$$
		\frac 1 {t^a \langle t^2-r^2 \rangle^{3/4-a/2}}.
		$$

		Regarding (\ref{second}), the computations are almost identical with those of Case I, including the cancellation with
		$$
		\int_{-\infty}^{-r} \frac 1 r \delta_r(\tau) |\tau-t|^{\alpha-1} \dd \tau = \frac 1 {r(r+t)^{1-\alpha}}.
		$$
		Only Case I applies, not Case II, leading for large $r$ to a bound of
		$$
		(\ref{second}) \les \frac {e^{-r}} {t^{1-a} \sqrt r} + \frac {e^{-r}} {r^{3/2-a}} \les \frac {e^{-r}} {t^{1-a} \sqrt r}.
		$$
		We are done.
	\end{proof}
	
	The explicit formulas of Proposition \ref{fractional_powers} for the size of the integral kernels of $E_\alpha$ allow us to compute their norms of interest.
	
	\begin{lemma}\lb{fractional_bounds} Suppose $\alpha=a+ib$, $a, b \in \R$, $0 \leq a=\re \alpha <1$. Then the integral kernel of $E_{(1+\alpha)/2}$ has size
		\be\lb{L1}
		\|E_{(1+\alpha)/2}(r, t)\|_{L^1_t} \les
		\left\{\begin{aligned}
			&r^{-1/2} \text{ when } r>1 \\
			&r^{a-1} \text{ when } r<1.
		\end{aligned}\right.
		\ee
		
		For $r<1$, the $t \in [0, r)$ portion of the integral kernel is bounded in $L^p_t$ for $p<\frac 1 {1-a}$ and in $L^{\frac 1 {1-a}, \infty}_t$ and has size
		\be\lb{outside_cone}
		\|\chi_{t \in [0, r)} E_{(1+\alpha)/2}(x, t)\|_{L^p_t} \les r^{a-2+1/p}
		\ee
		and is unbounded in $L^p_t$ for $p>\frac 1 {1-a}$.
		
		The $t \in [0, r)$ portion is exponentially small for $r>1$, being of size $e^{-r} r^{a-3/2+1/p}$ in $L^p_t$ for $p<\frac 1 {1-a}$ and in $L^{\frac 1 {1-a}, \infty}_t$, and unbounded in $L^p_t$ for $p>\frac 1 {1-a}$.
		
		
		Concerning the region inside the light cone, if $0\leq a < 1/2$ and $p<\frac 1 {1-a}$,
		$$
		\|\chi_{r<t} E_{(1+\alpha)/2}(r, t)\|_{L^p_t} \les r^{1/p-3/2}.
		$$
		In the limiting case $p=\frac 1 {1-a}$, $\|\chi_{r<t} E_{(1+\alpha)/2}(r, t)\|_{L^{\frac 1 {1-a}, \infty}_t} \les r^{-a-1/2}$.
		
		If $1/2<a<1$, the $t\in(r, \infty)$ portion has size
		$$
		\|\chi_{r<t} E_{(1+\alpha)/2}(r, t)\|_{L^p_t} \les \left\{\begin{aligned}
		&r^{1/p-3/2},&&p<\frac 4 {3-2a},\\
		&r^{-a-1/p},&&\frac 4 {3-2a}<p<\frac 1 {1-a}
		\end{aligned}\right.
		$$
		for $r>1$ and is uniformly bounded for $r<1$.
		
		In the limiting case $p=\frac 4 {3-2a}$, $\|E_{(1+\alpha)/2}(x, t)\|_{L^{p, \infty}_t} \les r^{-3/4-a/2}$.
		
		Suppose $1<a<3/2$. Then the $t \in [0, r)$ piece still has the same size
		(\ref{outside_cone}) and the $r<t$ piece has size
		$$
		\|\chi_{r<t} E_{(1+\alpha)/2}(r, t)\|_{L^p_t} \les \left\{\begin{aligned}
		&r^{1/p-3/2},&&p<\frac 4 {3-2a},\\
		&r^{-a-1/p},&&\frac 4 {3-2a}<p \leq \infty.
		\end{aligned}\right.
		$$
		
		For $3/2 \leq a<5/2$ and $1 \leq p \leq \infty$, again the portion outside the light cone has size (\ref{outside_cone}), while
		$$
		\|\chi_{r<t} E_{(1+\alpha)/2}(r, t)\|_{L^p_t} \les r^{1/p-3/2}
		$$
		for all $1 \leq p \leq \infty$.
	\end{lemma}
	
	\begin{observation}\lb{obs1}
		For the free wave equation, the corresponding kernels of order $(1+\alpha)/2$ have $2-a$ powers of decay.
		
		
		For the Klein--Gordon equation, the situation is more complicated. In general, estimates are inhomogeneous.
		
		Estimates with $2-a$ powers of decay hold only for the region $r>t$ (outside the light cone): for example, $\|\chi_{t \in [0, r)} E_{(1+\alpha)/2}\|_{L^{\frac 1 {1-a}, \infty}_t} \les r^{-1}$.
		
		
		Inside the light cone, when $0 \leq a \leq 1$, $E_{(1+\alpha)/2}$ has $3/2$ powers of decay in $L^p$ for $p \in [1, \min(\frac 4 {3-2a}, \frac 1 {1-a})]$ and, for $a \geq \frac 1 2$, it has $a+2/p$ powers of decay for $p \in (\frac 4 {3-2a}, \frac 1 {1-a})$, while having infinite norm in $L^p$ for $p>\frac 1 {1-a}$.
		
		Same is true for $\alpha \in [1, 3/2)$, with $\frac 1 {1-a}$ replaced by  infinity (so the kernel is bounded in all $L^p$).
		
		When $3/2 \leq a < 5/2$, $E_{(1+\alpha)/2}$ has exactly $\frac{3}{2}$ powers of decay for all $p \in [1, \infty]$.
		
		Thus, for the purpose of complex interpolation, there is no further decay gain in the above estimates from taking $\alpha>3/2$.
		
		A significant endpoint estimate is that, for $\alpha=3/2$, $\|E_{5/4}\|_{L^\infty_t} \les r^{-3/2}$.
		
		Another significant estimate, which also holds for the wave equation, for which the powers of decay are the same both inside and outside the light cone, is
		$$
		\|E_{3/4}\|_{L^p_t} \les r^{1/p-3/2}
		$$
		for $p<2$ and
		$$
		\|E_{3/4}\|_{L^{2, \infty}_t} \les r^{-1}.
		$$
	\end{observation}
	
	\begin{proof}
		Suppose $0<a<1$ and let us first compute the $L^1_t$ norm.
		
		There are several cases to consider. For $r>1$, the contribution of the interval $t \in [0, r)$ is exponentially small.
		
		On the interval $[r, \infty)$, the following integral converges and by scaling can be bound by
		$$
		\int_r^\infty \frac {dt} {t^a (t^2-r^2)^{3/4-a/2}} \les r^{-1/2}.
		$$
		
		For $r<1$, the contribution of the interval $t \in [0, r]$ is of order $r^{a-1}$ and the contribution of the interval $t \in [r, \infty)$ has size
		\be\lb{smallr}
		\int_r^1 dt + \int_1^\infty t^{-3/2} \dd t \les 1.
		\ee
		
		Next, we compute the $L^p_t$ norm of the kernel, $1 \leq p \leq \infty$. For small $r$ and, necessarily, $p<\frac 1 {1-a}$, on $t \in [0, r)$ the norm is bounded by
		$$
		\bigg( \int_0^r \frac {dt} {r^p(r-t)^{p(1-a)}} \bigg)^{1/p} \les (r^{pa-2p+1})^{1/p} = r^{a-2+1/p}.
		$$
		At the endpoint $p=\frac 1 {1-a}$, we obtain by rescaling that the $L^{\frac 1 {1-a}, \infty}_t$ norm is bounded by $r^{-1}$ on this interval.
		
		For large $r>1$ and $p<\frac 1 {1-a}$, for $t \in [0, r)$ any $L^p_t$ norm is exponentially small: for $p<\frac 1 {1-a}$ and for $p=\frac 1 {1-a}$ in $L^{\frac 1 {1-a}, \infty}_t$
		$$
		\bigg\| \frac {e^{-r}}{t^{1-a} \sqrt r} \bigg\|_{L^p_t([0, r])} \les e^{-r} r^{a-3/2+1/p}.
		$$
		
		On the interval $t \in (r, \infty)$, we have to distinguish two cases, according to whether $3/4-a/2<1/p$ or not. If $3/4-a/2<1/p$, then due to scaling
		$$
		\bigg\|\chi_{t \geq r} \frac 1 {t^a(t^2-r^2)^{3/4-a/2}}\bigg\|_{L^p_t} \les r^{1/p-3/2}.
		$$
		At the endpoint $3/4-a/2=1/p$, so $p=4/(3-2a)$, in order to avoid logarithms, we use the weak-type norm: by scaling
		$$
		\bigg\|\chi_{t \geq r} \frac 1 {t^a(t^2-r^2)^{3/4-a/2}}\bigg\|_{L^{p,\infty}_t} \les r^{1/p-3/2} = r^{-3/4-a/2}.
		$$
		In the other case $3/4-a/2>1/p$, we evaluate the norm as follows:
		$$\begin{aligned}
		\bigg\|\frac 1 {t^a}\bigg\|_{L^p_t([r, r+\rho])} &\les \rho^{1/p} r^{-a},\\
		\bigg\|\frac 1 {r^{3/4+a/2}(t-r)^{3/4-a/2}}\bigg\|_{L^p_t([r+\rho, 2r])} &\les r^{-3/4-a/2} \rho^{a/2-3/4+1/p},\\
		\|t^{-3/2}\|_{L^p_t([2r, \infty))} &\les r^{1/p-3/2}.
		\end{aligned}$$
		
		Choosing $\rho=r^{-1}$, we get an overall bound of $r^{-a-1/p}$, as $-a-1/p>1/p-3/2$ in this case. At the endpoint $3/4-a/2=1/p$, this is the same value as above.
		
		Same as in (\ref{smallr}), for small $r$ the contribution of the interval $[r, \infty)$ to the $L^p_t$ norm is uniformly bounded.

		Finally, we consider the case that $\frac{3}{2}\leq a<\frac{5}{2}$.
		When $t>r$, one has
		\begin{align*}
		\left|\chi_{r<t}E_{\left(1+\alpha\right)/2}\right| & \lesssim\frac{1}{t^{a}\left\langle t^{2}-r^{2}\right\rangle ^{\frac{3}{4}-\frac{a}{2}}}\\
		& =\frac{\left\langle t^{2}-r^{2}\right\rangle ^{\frac{a}{2}-\frac{3}{4}}}{t^{\alpha}}\\
		& \lesssim\frac{t^{a-\frac{3}{2}}}{t^{a}}.
		\end{align*}
		Therefore
		\[
		\left\Vert \chi_{r<t}E_{\left(1+\alpha\right)/2}\right\Vert _{L_{t}^{p}\left([r,\infty)\right)}\lesssim r^{\frac{1}{p}-\frac{3}{2}}.
		\]
		We are done.
	\end{proof}
	
	\begin{observation}\lb{cor:reversedE} Integrating in time, the $L^1_t$ norm (\ref{L1}) bounds the $L^\infty_t$ norm for half more powers of the Laplacian.
		
		For $\alpha=a+ib$, $a, b \in \R$, $0<a=\re\alpha<1$,
		\be\begin{aligned}\lb{epsilon}
			\|E_{1+\alpha/2}(r, t)\|_{L^\infty_t} \les \left\{\begin{aligned}
				&r^{-1/2}, && r>1 \\
				&r^{a-1}, && r<1.
			\end{aligned}\right.
		\end{aligned}\ee
		
		This is just checking for consistency and not as strong a result as the one we obtained by computing the $L^\infty$ norm directly, above.
		
		At the endpoint $\Re \alpha=0$ we have already obtained an $r^{-1}$ bound in $L^\infty_t$ for $C_1$, in Lemma \ref{c1bounds}.
		%
	\end{observation}
	
	%
	%
	
	\subsection{Klein--Gordon sine propagator bounds}
	
	We prove reversed Strichartz bounds for $S_{1/2}$, the Klein--Gordon sine propagator, by a $T T^*$ argument, starting from the corresponding bounds on $C_1$.
	
	$S_1$, being the Hilbert transform of $C_1$, is slightly worse behaved, but fortunately unimportant here. Due to fractional integration, bounds on $E_\alpha$ for $\alpha<1$ are also relevant.
	
	For $1/2 \leq a = \Re \alpha < 3/4$, the $L^1_t$ norm of $E_\alpha$ cannot be bounded by only one power of $r$. For example, when $\alpha=1/2$ we need to bound $\|E_{1/2}\|_{L^1_t}$ using $r^{-1/2}$ for $r>1$ and $r^{-1}$ for $r<1$.
	
	On the other hand, when $3/4 \leq a=\re \alpha < 1$, $E_\alpha$ has the same mapping properties as the corresponding operator for the free wave equation and some additional ones. In this region, $\|E_\alpha\|_{L^1_t} \sim \min(r^{2a-2}, r^{-1/2})$.
	
	By comparison, due to scaling, the same norm for the free wave equation is always bounded by $r^{2a-2}$.
	
	Note the particular role played by $E_{3/4}$, whose $L^p_t$ norm is uniformly bounded exactly by $r^{1/p-3/2}$, same as for the free wave equation.
	
	%
	%
	
	The corresponding inhomogeneous estimates can be stated more generally, but we need dual exponents for the $T T^*$ argument.
	
	One can also strengthen these estimates by using the Kato class and its dual, just as in Beceanu-Goldberg \cite{becgol}.
	
	\begin{theorem}\lb{thm:inhomoRev} For $\Re \alpha \geq 3/4$, the Klein--Gordon evolution operator $E_\alpha$ defined in Section \ref{notation} is bounded between the following spaces:
		$$
		E_\alpha: L^{p_1}_x L^q_t \to L^{p_2}_x L^q_t,\ \frac 1 {p_2} - \frac 1 {p_1} + 1 \in \bigg[\frac {2-2a}{3}, \frac 1 6\bigg],\ 1 \leq q \leq \infty,
		$$
		and
		$$
		E_{\alpha+1/2}: L^{p_1}_x L^1_t \to L^{p_2}_x L^\infty_t,\ \frac 1 {p_2} - \frac 1 {p_1} + 1 \in \bigg[\frac {2-2a}{3}, \frac 1 6\bigg].
		$$
	\end{theorem}
	The first endpoint is the same as for the free wave equation; the other is specific to the Klein--Gordon equation.

	\subsection{Fractional powers}

	For the free wave equation, purely imaginary powers of the Hamiltonian $(-\Delta)^{i\sigma}$ are bounded operators on $L^p$, $1 < p < \infty$, because they are Mihlin multipliers.
	
	Same is true in the Klein--Gordon case: $(-\Delta+1)^{i\sigma}$ has the same properties as $(-\Delta)^{i\sigma}$, because it is also a Mihlin multiplier.
	
	However, $(-\Delta+1)^{-\alpha}$ has strictly better properties than $(-\Delta)^{-\alpha}$ when $\Re \alpha > 0$, because $(-\Delta+1)^{-\alpha}$ has an integrable kernel, hence is also bounded on $L^1$ and $L^\infty$.
	
	Indeed, Klein--Gordon negative powers have all the fractional integration properties of $(-\Delta)^{-\alpha}$, but are additionally bounded for non-sharp pairs of powers:
	\begin{proposition} For $a = \Re \alpha \geq 0$, $(-\Delta+1)^{-\alpha}$ is bounded on $L^p$, $p \in (1, \infty)$. For $a = \re \alpha>0$, whenever $\frac 1 p \geq \frac 1 q \geq \frac 1 p - \frac {2a} 3$, one has that
		$$
		(-\Delta+1)^{-\alpha} \in \B(L^p, L^q)
		$$
		with the usual modifications at the endpoints.
	\end{proposition}
	\begin{proof}
		For $\alpha > 0$, $(-\Delta+1)^{-\alpha}$ is given by convolution with the inverse Fourier transform of $(|\xi|^2+1)^{-\alpha}$.
		
		Since $(|\xi|^2+1)^{-\alpha}$ is an analytic function with polynomial decay at infinity, where it decays like $|\xi|^{-2\alpha}$, its inverse Fourier transform is a rapidly decaying function, which is smooth except for a singularity of $|x|^{2\alpha-3}$ at zero.
		
		Thus
		$$
		\mc F^{-1} (|\xi|^2+1)^{-\alpha} \in L^1 \cap L^{\frac {2\alpha}3-1, \infty}
		$$
		and the boundedness properties follow from Young's inequality.
	\end{proof}	
	
	The fractional integration kernels $(-\Delta+1)^{-\alpha}$, $\Re \alpha > 0$, are bounded in absolute value, so they also act as stated on the reversed Strichartz norms. Therefore, if a reversed Strichartz estimate holds for some exponent $\alpha$, we can increase the real part of $\alpha$ without prejudice:
	\begin{corollary}\lb{increase}
		If some reversed Strichartz estimate holds for $E_{\alpha}$ with a certain $\alpha$, then it also holds for all $\beta$ with $\Re \beta > \Re \alpha$.
	\end{corollary}

	\subsection{The free sine propagator}\lb{subsec:sinpro}
	In order to introduce nonlinear potentials into the equation (i.e.\;to use our results in the study of semilinear Klein--Gordon equations), we need to bound the free sine propagator from Duhamel's formula (\ref{duhamel}).
	
	We split the inhomogenous term into two parts, the wave part
	$$\begin{aligned}
	(S_\Delta F)(t) &:= \int_{-\infty}^t \frac {\sin((t-s)\sqrt{-\Delta})} {\sqrt{-\Delta}} F(s) \dd s \\
	&= \int_{-\infty}^t \frac 1 {4 \pi (t-s)} \delta_{t-s}(|x|) \ast F(s) \dd s
	\end{aligned}$$
	and the Bessel part
	$$
	(S_B F)(t) := \int_{-\infty}^t \frac 1 {4\pi} \frac {J_1(\sqrt{(t-s)^2+|x|^2})}{\sqrt{(t-s)^2+|x|^2}} \ast F(s) \dd s.
	$$
	The wave part is well understood, see \cite{becgol}. Thus, it is left for us to examine the Bessel part.
	
	\begin{lemma}\lb{waveandBessel} The Bessel part of the propagator satisfies the following bounds:
		$$
		\|S_B\|_{L^p_t} \les \left\{\begin{aligned}
		&r^{1/p-3/2},&&1 \leq p<4/3,\\
		&r^{-1/p},&&1 \leq 4/3<p \leq\infty,
		\end{aligned}\right.
		$$
		while for $p=4/3$
		$$
		\|S_B\|_{L^{4/3, \infty}_t} \les r^{-3/4}.
		$$
	\end{lemma}
	The important norms here are
	\be\lb{important_norms}
	|x|^{-1/2} L^\infty_x L^1_t \subset L^{6, \infty}_x L^1_t,\ |x|^{-3/4} L^\infty_x L^{4/3, \infty}_t \subset L^{4, \infty}_x L^{4/3, \infty}_t,\ L^\infty_{x, t}.
	\ee
	All the other norms can be obtained by interpolation between these three.
	
	The first two norms have the specific Klein--Gordon scaling, $3/2$ powers of decay, while the last has no decay.
	
	This family of estimates contains no bound with the wave equation scaling, i.e.\;two powers of decay. 
	
	\begin{proof}
		For the Bessel part $S_B$, the $L^2_t$ norm of the Bessel propagator is $\les r^{-1/2}$:
		$$
		\frac {J_1(\sqrt{t^2-r^2})}{\sqrt{t^2-r^2}} \les \frac 1 {\langle t^2-r^2\rangle^{3/4}},
		$$
		so
		$$
		\| S_B \|_{L^2_t}^2 \les \int_r^{r+\delta} 1 \dd t + \int_{r+\delta}^\infty \frac {dt}{(t^2-r^2)^{3/2}} \les \delta+\frac 1 {\delta^{1/2}r^{3/2}} \les \frac 1 r.
		$$
		Here we set $\delta=r^{-1}$. More generally, for $p>4/3$ the $L^p_t$ norm is bounded by $r^{-1/p}$: again setting $\delta = r^{-1}$
		$$
		\| S_B \|_{L^p_t}^p \les \int_r^{r+\delta} 1 \dd t + \int_{r+\delta}^\infty \frac {dt}{(t^2-r^2)^{3p/4}} \les \delta+\frac 1 {\delta^{3p/4-1}r^{3p/4}} \les \frac 1 r.
		$$
		This has $2/p$ powers of decay ($1/p$ in time and $1/p$ in space), which would be the correct scaling when $p=1$ for the free wave propagator.
		
		However, for $p<4/3$ we get a different scaling for the Bessel part:
		$$
		\|S_B\|_{L^p_t}^p \les \int_r^\infty \frac {dt}{(t^2-r^2)^{3p/4}} \les r^{1-3p/2}.
		$$
		Hence the $\|S_B\|_{L^p_t}$ norm is bounded by $r^{1/p-3/2}$ in this case, so there are no more than $3/2$ powers of decay in total. In particular, the $L^1_t$ norm is bounded by $r^{-1/2}$. 
		
		At the endpoint $p=4/3$, we have to use a weak-type estimate in order to avoid logarithms:
		$$
		\|S_B\|_{L^{4/3, \infty}_t} \les r^{-3/4},
		$$
		so the kernel is in $L^{4, \infty}_x L^{4/3, \infty}_t$.
	\end{proof}
	
	\subsection{Pointwise decay estimates}
	Another important class of dispersive estimates is that of pointwise decay estimates.
	
	For the wave equation, the two best-known pointwise decay estimates are (see for example \cite{becgol})
	$$
	\bigg\|\frac {\sin(t\sqrt{-\Delta})}{\sqrt{-\Delta}} f\bigg\|_{L^\infty_x} \les |t|^{-1} \|\nabla f\|_{L^1},\ \|\cos(t\sqrt{-\Delta}) f\|_{L^\infty_x} \les |t|^{-1} \|\Delta f\|_{L^1}.
	$$
	
	In this section we replicate these estimates and their proofs for the Klein--Gordon equation.
	
	
	\begin{theorem}\lb{thm:fkgpointwise} Consider the free Klein--Gordon equation in $\R^{3+1}$. Same as for the wave equation,
		$$
		\|S_{1/2}(t) f\|_{L^\infty_x} \les |t|^{-1} \|\nabla f\|_{L^1},\ \|C_1(t) f\|_{L^\infty_x} \les |t|^{-1} \|f\|_{L^1}.
		$$
	\end{theorem}
	%
	%
	\begin{proof} In order to evaluate the decay of the translation-invariant sine propagator, we integrate by parts:
		$$\begin{aligned}
		\bigg(\frac {\sin(t\sqrt{-\Delta+1})}{\sqrt{-\Delta+1}} f\bigg)(0, t) &= \frac 1 {4\pi t} \int_{|y|=t} f(y) \dd y - \frac 1 {4\pi} \int_{|y| \leq t} \frac {J_1(\sqrt{t^2-r^2})}{\sqrt{t^2-r^2}} f(y) \dd y \\
		&= \frac 1 {4\pi} \int_{S^2} f(t, \omega) t - \int_0^t \frac {J_1(\sqrt{t^2-r^2})}{\sqrt{t^2-r^2}} f(r, \omega) r^2 \dd r \dd \omega \\
		&= \frac 1 {4\pi} \int_{S^2} \int_0^\infty K_t(r) f_r(r, \omega) \dd r \dd \omega \\
		&\les \sup_r \bigg|\frac {K_t(r)}{r^2}\bigg| \|f_r\|_{L^1},
		\end{aligned}$$
		where $K_t$ is the following integral kernel that we have to compute:
		\be\lb{first_expression}
		K_t(r) = \int_0^r \frac {J_1(\sqrt{t^2-\rho^2})}{\sqrt{t^2-\rho^2}} \rho^2 \dd \rho
		\ee
		for $r<t$ and
		\be\lb{second_expression}
		K_t(r) = \int_0^t \frac {J_1(\sqrt{t^2-\rho^2})}{\sqrt{t^2-\rho^2}} \rho^2 \dd \rho - t
		\ee
		for $r\geq t$; the latter expression is actually independent of $r$.
		
		For (\ref{first_expression}), since $|J_1(z)/z| \les 1$, there is always a bound of $r^3$ (or $r$ after dividing by $r^2$), which can be further refined to
		\be\lb{refinement}
		r \sup_{\rho \in [0, r]} \frac {J_1(\sqrt{t^2-\rho^2})}{\sqrt{t^2-\rho^2}} \les \frac r {\sqrt{t^2-r^2} \langle t^2-r^2\rangle^{1/4}}.
		\ee
		If $r<t/2$, then this is bounded by $r/t^{3/2}$, so by $t^{-1/2}$.
		
		After the substitution $\sigma=\sqrt{t^2-\rho^2}$ we also get
		$$\begin{aligned}
		(\ref{first_expression}) &=\int_{\sqrt{t^2-r^2}}^t J_1(\sigma) \sqrt{t^2-\sigma^2} \dd \sigma \\
		&= J_0(\sqrt{t^2-r^2}) r + \int_{\sqrt{t^2-r^2}}^t J_0(\sigma) \partial_\sigma \sqrt{t^2-\sigma^2} \dd \sigma \\
		&\les \sup_{\sigma \in [\sqrt{t^2-r^2}, t]} |J_0(\sigma)| r \les \frac{r}{\langle t^2-r^2 \rangle^{1/4}}.
		\end{aligned}$$
		Hence $(\ref{first_expression})/r^2$ is also bounded by
		$$
		\frac {1}{r\langle t^2-r^2\rangle^{1/4}}.
		$$
		and taking the geometric mean of this bound and (\ref{refinement}) we get
		$$
		\frac 1 {(t^2-\rho^2)^{1/4} \langle t^2-r^2\rangle^{1/4}}.
		$$
		This is less than $C/t$ whenever $r<t/2$, while the prior bound is less than $C/t$ when $t/2 \leq r < t$.
		
		For the second expression (\ref{second_expression}), the main term $J_0(0)t=t$ cancels and we are left with
		$$
		\int_0^t J_0(\sigma) \partial_\sigma \sqrt{t^2-\sigma^2} \dd \sigma.
		$$
		Since $J_0$ is bounded, this expression is bounded by $t$, then $t/r^2 \leq 1/t$ in fact, after dividing by $r^2$.
		
		For $t>1$, one can do better: dividing the integration region into two portions and integrating by parts on the lower first portion, we get
		$$
		\begin{aligned}
		&\int_0^{\sqrt{t^2-\rho^2}} \frac {\sigma J_0(\sigma) \dd \sigma}{\sqrt{t^2-\sigma^2}} + \int_{\sqrt{t^2-\rho^2}}^t \frac {\sigma J_0(\sigma) \dd \sigma}{\sqrt{t^2-\sigma^2}} \les \\
		&\les \frac {\sqrt{t^2-\rho^2} J_1(\sqrt{t^2-\rho^2})}{\rho} + \int_0^{\sqrt{t^2-\rho^2}} \frac {\sigma^2J_1(\sigma) \dd \sigma} {(t^2-\sigma^2)^{3/2}} + \frac \rho {\langle t^2-\rho^2 \rangle^{1/4}} \\
		& \les \frac {\sqrt{t^2-\rho^2}}{\langle t^2-\rho^2 \rangle^{1/4} \rho} + \frac \rho {\langle t^2-\rho^2 \rangle^{1/4}} \les 1,
		\end{aligned}
		$$
		by setting $\rho=\sqrt t$. This leads to a bound of $1/r^2 \leq 1/t^2$ for (\ref{second_expression}).
		
		%
		%
		
		The cosine kernel $C_1$ has size $\|C_1(t)\|_{L^\infty_r} \les 1/t$, same as for the wave equation, as proved by our previous estimates.
		%
		%
		%
	\end{proof}
	
		Specifically to Klein--Gordon equation, using Proposition \ref{fractional_powers}, we also obtain an important pointwise decay estimate for the endpoint with $\alpha=3/2$ of the fractional power of Klein--Gordon evolution. 	This is vital in the proof of the endpoint standard Strichartz estimate later on.
		We also record the poinwise decay for $\alpha=1/2$.
		\begin{theorem}\lb{KGpointDecay} Consider the free Klein--Gordon equation in $\R^{3+1}$. One has for $b\in\mathbb{R}$
			\[
			\left|E_{1+ib}\right|\left(r,t\right)\lesssim\frac{1}{|t|},\quad	\left|E_{\frac{5}{4}+ib}\right|\left(r,t\right)\lesssim\frac{1}{|t|^{\frac{3}{2}}}.
			\]
		\end{theorem}
		\begin{proof}
			We consider $t\geq 0$ and focus on the second estimate above. Given $\alpha= a+ib$, recall that $0\leq t<r$ 
			\[
			\left|E_{\frac{1+\alpha}{2}}\right|\left(r,t\right)\lesssim\frac{e^{-r}}{t^{1-a}\sqrt{r}}
			\]
			when $r$ is large and 
			\[
			\left|E_{\frac{1+\alpha}{2}}\right|\left(r,t\right)\lesssim\frac{1}{r\left(r-t\right)^{1-a}}
			\]
			when $r\leq 1$ is small.  Setting $a=\frac{3}{2}$, one has the desired estimate in the first region. By construction, if $t> 1$, the second region does not exist. For $0\leq t\leq 1$, after plugging $a=\frac{3}{2}$, it is clear that $\frac{1}{r}<\frac{1}{t^{\frac{3}{2}}}$. 
			
			Next, for $t>r$
			\[
			\left|E_{\frac{1+\alpha}{2}}\right|\left(r,t\right)\lesssim\frac{1}{t^{a}\left\langle t^{2}-r^{2}\right\rangle ^{\frac{3}{4}-\frac{a}{2}}}.
			\]
			Taking $a=\frac{3}{2}$, we obtain the desired estimate in this region.
		\end{proof}

	\section{The Klein--Gordon equation with a scalar potential}\lb{KGP}
	In this section, consider the linear Klein--Gordon equation with a time-independent potential (\ref{pkg}):
	$$
	u_{tt} - \Delta u + u + V(x) u = F,\ u(0)=u_0,\ u_t(0)=u_1.
	$$
	
	We extend all estimates obtained for the free Klein--Gordon equation in Section \ref{FKG} to the perturbed equation (\ref{pkg}) using the spectral representation of the fundamental solution and results from \cite{becgol}.
	
	The following computation will help us better understand the difference between the Bessel and wave portion of the Klein--Gordon sine propagator.
	\begin{lemma}\lb{spectral} For $r, t \in \R$, $t \ne 0$,
		$$
		\frac 1 {2|t|} \int_{-\infty}^\infty \big(\cos(t\sqrt{\lambda^2+1}) - \cos(t\lambda)\big) e^{i\lambda r} \dd \lambda = - \chi_{r \leq |t|} \frac {J_1(\sqrt{t^2-r^2})}{\sqrt{t^2-r^2}}.
		$$
	\end{lemma}
	\begin{proof} With no loss of generality, assume $t>0$. By spectral calculus, the free Klein--Gordon sine propagator is given by the formula
		\begin{multline}\lb{spectral_calc}
		\frac 1 {2\pi i} \int_0^\infty \frac {\sin(t\sqrt{\lambda^2+1})}{\sqrt{\lambda^2+1}} (R_{0+}(\lambda^2)-R_{0-}(\lambda^2)) \lambda \dd \lambda = \\
		=\frac {\sgn t} {4\pi} \bigg(\frac {\delta_r(t)} r - \chi_{r \leq |t|} \frac {J_1(\sqrt{t^2-r^2})}{\sqrt{t^2-r^2}}\bigg).
		\end{multline}
		The density of the spectral measure of the free Laplacian appearing in this formula is
		$$
		\frac 1 {2\pi i} (R_{0+}(\lambda^2)-R_{0-}(\lambda^2)) = \frac {\sin(\lambda r)}{4\pi r}.
		$$
		Consequently, identifying the two expressions in (\ref{spectral_calc}) pointwise, we obtain that for all $r \in \R$ and $t \ne 0$
		\be\begin{aligned}\lb{identity}
			(\sgn t)\bigg(\frac {\delta_{|t|}(r)} {|t|} - \frac {J_1(\sqrt{t^2-r^2})}{\sqrt{t^2-r^2}}\bigg) &= \int_0^\infty \frac {\sin(t\sqrt{\lambda^2+1})}{\sqrt{\lambda^2+1}} \frac {\sin(\lambda r)} r \lambda \dd \lambda \\
			&=
			\frac 1 {2i} \int_{-\infty}^\infty \frac {\sin(t\sqrt{\lambda^2+1})}{\sqrt{\lambda^2+1}} \frac {e^{i\lambda r}}{r} \lambda \dd \lambda \\
			&= \frac 1 {2t} \int_{-\infty}^\infty \cos(t\sqrt{\lambda^2+1}) e^{i\lambda r} \dd \lambda,
		\end{aligned}\ee
		in the sense of the Fourier transform of tempered distributions.
		
		
		Doing the same computation for the free wave equation sine propagator leads instead to
		$$
		\frac 1 {2t} \int_{-\infty}^\infty \cos(t|\lambda|) e^{i\lambda r} \dd \lambda = (\sgn t) \frac {\delta_{|t|}(r)} {|t|}.
		$$
		Taking the difference, we have obtained that
		$$
		\frac 1 {2t} \int_{-\infty}^\infty \big(\cos(t\sqrt{\lambda^2+1}) - \cos(t\lambda)\big) e^{i\lambda r} \dd \lambda = - (\sgn t) \chi_{r \leq |t|} \frac {J_1(\sqrt{t^2-r^2})}{\sqrt{t^2-r^2}}.
		$$
		This is the desired conclusion.
	\end{proof}
	
	Now we prove reverse Strichartz estimates for perturbed Hamiltonians $H=-\Delta+V$. We lose the translation invariance, but otherwise the estimates are the same as in the free case.
	
	Just as in the free case, the wave part of the Klein--Gordon propagator is already understood, in this case due to \cite{becgol}. Thus, we are left with analyzing the difference between the Klein--Gordon and wave sine propagators, which we call the Bessel part of the propagator by analogy with the free equation.
	
	This part scales differently from the wave part, hence satisfies different bounds.
	
	\begin{theorem}\lb{SBH} Consider the Klein--Gordon equation in $\R^{3+1}$ with Hamiltonian $H=-\Delta+V$, $V \in L^{3/2, 1}$, such that $H$ has no threshold bound states (eigenstates or resonance). Then the difference of kernels (the Bessel part)
		$$
		S_B^H(t) = \frac {\sin(t\sqrt{H+1})P_c}{\sqrt{H+1}} - \frac {\sin(t\sqrt{H})P_c}{\sqrt{H}}
		$$
		is bounded in $L^\infty_{x, y, t}$ and
		$$
		\|S_B^H(t)\|_{L^1_t} \les |x-y|^{-1/2},\ \|S_B^H(t)\|_{L^{4/3, \infty}_t} \les |x-y|^{-3/4},\ \|S_B^H(t)\|_{L^\infty_t} \les 1.
		$$
		Consequently $S_B^H$ is dominated by a convolution operator, with convolution kernel in
		$$
		|x|^{-1/2} L^\infty_x L^1_t \subset L^{6, \infty}_x L^1_t,\ |x|^{-3/4} L^\infty_x L^{4/3, \infty}_t \subset L^{4, \infty}_x L^{4/3, \infty}_t,\ L^\infty_{x, t}.
		$$
	\end{theorem}
	These are the three most important sine propagator bounds we obtained in the free Klein--Gordon case, see Lemma \ref{waveandBessel}, from which one can recover all the others by interpolation.
	\begin{proof}
		Taking the difference between the Klein--Gordon and wave sine propagators, we are left with
		$$\begin{aligned}
		&S_B^H(t) = \frac {\sin(t\sqrt{H+1})P_c}{\sqrt{H+1}} - \frac {\sin(t\sqrt{H})P_c}{\sqrt{H}} = \\
		&=\frac 1 {2\pi} \int_0^\infty \bigg(\frac {\sin(t\sqrt{\lambda^2+1})}{\sqrt{\lambda^2+1}} - \frac {\sin(t \lambda)}{\lambda} \bigg) (R_{V+}(\lambda^2)-R_{V-}(\lambda^2))P_c \lambda \dd \lambda \\
		&= \frac 1 {2\pi} \int_{-\infty}^\infty \bigg(\frac {\sin(t\sqrt{\lambda^2+1})}{\sqrt{\lambda^2+1}} - \frac {\sin(t \lambda)}{\lambda}\bigg) R_{V+}(\lambda^2)P_c \lambda \dd \lambda \\
		&= \frac 1 {2\pi t} \int_{-\infty}^\infty (\cos(t\sqrt{\lambda^2+1}) - \cos(t\lambda)) \partial_\lambda R_{V+}(\lambda^2)P_c \dd \lambda \\
		&= -\frac {\sgn t} \pi \int_{-|t|}^{|t|} \frac {J_1(\sqrt{t^2-\tau^2})}{\sqrt{t^2-\tau^2}} \big(\partial_\lambda R_{V+}(\lambda^2)P_c\big)^\vee(\tau) \dd \tau
		\end{aligned}$$
		by the previous Lemma \ref{spectral}.
		
		As shown in \cite{becgol}, the last factor is in fact
		\be\lb{last_factor}
		\big(\partial_\lambda R_{V+}(\lambda^2)P_c\big)^\vee(x, y, \tau) = \chi_{\tau \geq 0} \frac {\tau \sin(\tau\sqrt H)P_c}{\sqrt H}(x, y).
		\ee
		
		
		From \cite{becgol} we also know that this expression is bounded by
		\be\lb{crucial}
		\sup_{x, y} \|(\partial_\lambda R_{V+}(\lambda^2)P_c)^\vee(x, y, \tau)\|_{L^1_\tau} = \sup_{x, y} \|\tau (R_{V+}(\lambda^2)P_c)^\vee(x, y, \tau)\|_{L^1_\tau} < \infty.
		\ee
		
		In other words, for each $x, y \in \R^3$, the kernel $S_H^B(x, y, t)$ is an integrable combination of $\frac {J_1(\sqrt{t^2-\tau^2})}{\sqrt{t^2-\tau^2}}$ for various values of $\tau$ for which $|\tau| \leq |t|$.
		
		We thus obtain a bound of
		$$
		\sup_{x, y, t} \bigg|\frac{\sin(t\sqrt{H+1})P_c}{\sqrt{H+1}} - \frac{\sin(t\sqrt H)P_c}{\sqrt H}\bigg| \les \sup_{0 \leq \tau \leq t} \frac {J_1(\sqrt{t^2-\tau^2})}{\sqrt{t^2-\tau^2}} < \infty.
		$$
		This is one of the three main bounds we proved for the free Klein--Gordon propagator, see (\ref{important_norms}).

		A further fact we can use is that, due to the finite speed of propagation of the wave equation,
		$$
		\frac {\tau \sin(\tau\sqrt H)}{\sqrt H}(x, y)
		$$
		is supported on the region $|\tau| \geq |x-y|$. However, if $H$ has negative energy eigenstates, then the expression we are actually interested in,
		$$
		\frac {\tau \sin(\tau\sqrt H)P_c}{\sqrt H}(x, y),
		$$
		has nonzero tails in the region $|\tau| < |x-y|$, whose contribution we have to evaluate.
		
		We use this splitting to prove an $L^1_t$ bound for the kernel $S^B_H$. In the region $\tau \geq |x-y|$, since
		\be\lb{bessel_bound}
		\int_{|t| \geq \tau} \bigg|\frac {J_1(\sqrt{t^2-\tau^2})}{\sqrt{t^2-\tau^2}}\bigg| \dd t \les \int_{|t| \geq \tau} \frac {dt}{(t^2-\tau^2)^{3/4}} \les \tau^{-1/2}
		\ee
		and the expression (\ref{last_factor}) is integrable in $\tau$ as per (\ref{crucial}), we get a contribution of
		\begin{multline*}
		\bigg\|\int_{-t}^t \frac {J_1(\sqrt{t^2-\tau^2})}{\sqrt{t^2-\tau^2}} \chi_{\tau \geq |x-y|} \big(\partial_\lambda R_{V+}(\lambda^2)P_c\big)^\vee(\tau) \dd \tau\bigg\|_{L^1_t}
		\les \sup_{\tau \geq |x-y|} \tau^{-1/2} \leq \frac 1 {|x-y|^{1/2}}.
		\end{multline*}
		
		Next, we consider the contribution of the region outside the light cone. In this region, due to the finite speed of propagation of the overall solution, the continuous and point spectrum parts must exactly cancel out. Supposing there are $N$ eigenvalues $\lambda_1,\ldots,\lambda_N$, with normalized eigenstates $f_1,\ldots,f_N$,
		$$\begin{aligned}
		\chi_{|\tau|<|x-y|} \frac {\sin(\tau\sqrt H)P_c}{\sqrt H} &= -\chi_{|\tau|<|x-y|} \frac{\sin(\tau\sqrt H) P_p}{\sqrt H} \\
		&= -\chi_{|\tau|<|x-y|} \sum_{k=1}^N \frac {\sinh(\tau\sqrt{-\lambda_k})}{\sqrt{-\lambda_k}} f_k(x) \otimes \ov f_k(y).
		\end{aligned}
		$$
		Due to sharp Agmon bounds, see Appendix \ref{sec:agmon}, we get that for $0 \leq \tau < |x-y|$
		\be\lb{fine_bound}
		(R_{V+}(\lambda^2)P_c)^\vee(x, y, \tau) = \sum_{k=1}^N R_k,\ |R_k| \les \langle x \rangle^{-1} \langle y \rangle^{-1} e^{(\tau - |x| - |y|)\sqrt{-\lambda_k}},
		\ee
		hence
		$$
		|(\partial_\lambda R_{V+}(\lambda^2)P_c)^\vee(x, y, \tau)| \les \tau \langle x \rangle^{-1} \langle y \rangle^{-1} e^{(\tau - |x| - |y|)\sqrt{-\lambda_1}}.
		$$
		
		This leads to a uniform bound of $\langle x \rangle^{-1} \langle y \rangle^{-1}$, but we actually need to use formula (\ref{fine_bound}) under the integral. Doing this saves a power of $\tau$, as follows: for $\alpha, c>0$
		$$
		\int_0^T \tau^\alpha e^{c\tau} \dd \tau \les T^\alpha e^{cT}. 
		$$
		
		The $L^1_t$ norm contributed by the region outside the light cone is no more than
		$$
		\sum_{k=1}^N \int_0^{|x-y|} \frac {\tau^{1/2} e^{(\tau - |x| - |y|)\sqrt{-\lambda_k}} \dd \tau}{\langle x \rangle \langle y \rangle} \les \frac {|x-y|^{1/2}}{\langle x \rangle \langle y \rangle} \les \frac 1 {|x-y|^{1/2}}.
		$$
		Here the $\tau^{1/2}$ factor comes from (\ref{bessel_bound}) together with one power of $\tau$ corresponding to $\partial_\lambda R_V$.
		
		Consequently, the same bound holds for the whole expression:
		$$
		\|S_H^B\|_{L^1_t} = \bigg\|\frac{\sin(t\sqrt{H+1})P_c}{\sqrt{H+1}} - \frac{\sin(t\sqrt H)P_c}{\sqrt H}\bigg\|_{L^1_t} \les \frac 1 {|x-y|^{1/2}}.
		$$
		This is another of the three main estimates we got for the free Klein--Gordon sine propagator, see Lemma \ref{waveandBessel} and (\ref{important_norms}).
		
		To prove the remaining bound, firstly by rescaling we have that
		\be\lb{rescale}
		\bigg\|\chi_{t \geq \tau} \frac 1 {(t^2-\tau^2)^{3/4}}\bigg\|_{L^{4/3, \infty}_t} \les \tau^{-3/4}.
		\ee
		Since $L^{4/3, \infty}$ is a Banach space, in the region $\tau \geq |x-y|$ this together with (\ref{crucial}) leads to
		$$
		\bigg\|\frac{\sin(t\sqrt{H+1})}{\sqrt{H+1}} - \frac{\sin(t\sqrt H)}{\sqrt H}\bigg\|_{L^1_t} \les \sup_{\tau \geq |x-y|} \tau^{-3/4} \leq \frac 1 {|x-y|^{3/4}}.
		$$
		Outside the light cone we can prove the same bound using (\ref{fine_bound}):
		$$
		\sum_{k=1}^N \int_0^{|x-y|} \frac {\tau^{1/4} e^{(\tau - |x| - |y|)\sqrt{-\lambda_k}} \dd \tau}{\langle x \rangle \langle y \rangle} \les \frac {|x-y|^{1/4}}{\langle x \rangle \langle y \rangle} \les \frac 1 {|x-y|^{3/4}}.
		$$
		The $\tau^{1/4}$ factor comes from (\ref{rescale}) together with one power of $\tau$ corresponding to $\partial_\lambda R_V$.
		
		Putting the two contributions together, the bound holds for the whole expression. This is the last estimate we wanted to prove.
	\end{proof}
	
	The same method works for all reversed Strichartz norms and can be used to deduce all the reversed type estimates for the perturbed case from reversed type estimates for the free case.
	\begin{proposition}\lb{C1H} Consider the Klein--Gordon equation in $\R^{3+1}$ with Hamiltonian $H=-\Delta+V$, $V \in L^{3/2, 1}$, such that $H$ has no threshold bound states (eigenstates or resonance). Then the kernel of
		$$
		C_1^H(t) P_c = \frac{\cos(t\sqrt{H+1}) P_c}{H+1}
		$$
		satisfies the bounds of Lemma \ref{c1bounds}, with $r=|x-y|$.
		
		In addition, all the estimates in Lemma \ref{fractional_bounds}, Corollary \ref{cor:reversedE}, and Theorem \ref{thm:inhomoRev} hold for
		\begin{equation} E^{H}_{\alpha}(t)P_{c}=\frac{e^{it\sqrt{H+1}}P_{c}}{\left(H+1\right)^{\alpha}}.
		\end{equation}
	\end{proposition}
	\begin{proof}
		Starting from (\ref{identity}), we obtain
		\begin{align*}
		\frac{\sin(t\sqrt{H+1})P_{c}}{\sqrt{H+1}}(x, y) & =\frac{1}{2\pi t}\int_{-\infty}^{\infty}\cos(t\sqrt{\lambda^{2}+1})\partial_{\lambda}R_{V+}\left(\lambda^{2}\right)P_{c}\dd\lambda\\
		& =\int_{-\infty}^{\infty}S_{1/2}\left(r=\tau,t\right)(\partial_{\lambda}R_{V+}(\lambda^{2})P_c)^\vee(x, y, \tau)\dd\tau.
		\end{align*}
		Integrating from some fixed $t$ to infinity, this leads to 
		\[
		\frac{\cos(t\sqrt{H+1})P_{c}}{H+1}(x, y)=\int_{-\infty}^{\infty}C_{1}\left(r,t\right)(\partial_{\lambda}R_{V+}(\lambda^{2})P_c)^\vee(x, y, r)\dd r.
		\]
		
		Similarly, by means of fractional integration we obtain that
		\begin{equation}
		[E^{H}_{\alpha}\left(t\right)P_{c}](x, y) = \frac{e^{it\sqrt{H+1}}P_{c}}{\left(H+1\right)^{\alpha}}=\int_{-\infty}^{\infty}E_{\alpha}(r,t)\left(\partial_{\lambda}R_{V+}\left(\lambda^{2}\right)P_c\right)^\vee(x, y, r)\dd r.\label{eq:EaPerturbed}
		\end{equation}
		
		As in the proof of the previous theorem, we again split the integral into two regions, according to whether $r \geq |x-y|$ or $r<|x-y|$.
		
		In the first region (inside the light cone), the bound
		\begin{equation}\label{eq:boundkernel}
		\sup_{x,y}\big\| \left(\partial_{\lambda}R_{V+}(\lambda^{2}) P_c\right)^\vee(x, y, \tau)\big\|_{L_{\tau}^{1}}<\infty
		\end{equation}
		from \cite{becgol} implies that this portion of $C_1^H(t)$ is an integrable combination of $C_1(\tau, t)$, for various values of $\tau \geq |x-y|$.
		
		Since our estimates for the free Klein--Gordon equation are monotonically decreasing as a function of the radius, one can replace each value of $\tau$ by $|x-y|$ with no prejudice and obtain the same estimates as in the free case, with a different constant, in terms of $|x-y|$:
		$$\begin{aligned}
		&\bigg\|\int_{r \geq |x-y|}E_{\alpha}(r,t)\left(\partial_{\lambda}R_{V+}\left(\lambda^{2}\right)P_c\right)^\vee(x, y, r)\dd r\bigg\|_{L^p_t} \leq \\
		&\leq \sup_{r \geq |x-y|}\|E_{\alpha}(r,t)\|_{L^p_t} \ \sup_{x,y}\big\| \left(\partial_{\lambda}R_{V+}(\lambda^{2}) P_c\right)^\vee(x, y, \tau)\big\|_{L_{\tau}^{1}}.
		\end{aligned}$$
		
		In the region outside the light cone, one has the improved bound (\ref{fine_bound}) and the reasoning proceeds as in the proof of the previous theorem: in general
		$$
		\sum_{k=1}^N \int_0^{|x-y|} \frac {\tau^{1-\alpha} e^{(\tau - |x| - |y|)\sqrt{-\lambda_k}} \dd \tau}{\langle x \rangle \langle y \rangle} \les \frac {|x-y|^{1-\alpha}}{\langle x \rangle \langle y \rangle} \les \frac 1 {|x-y|^{\alpha}}.
		$$
	\end{proof}
	
	\begin{corollary}[Reversed Strichartz estimates]\lb{cor_partial_results} Consider the Klein--Gordon equation in $\R^{3+1}$ with Hamiltonian $H=-\Delta+V$, $V \in L^{3/2, 1}$, such that $H$ has no threshold bound states (eigenstates or resonance). Then the operators $E_{1/2}^H P_c$, $S_{1/2}^H P_c$, and $C_{1/2}^H P_c$ are bounded between the following spaces:
		$$
		S_{1/2}^H P_c, C_{1/2}^H P_c, E_{1/2}^H P_c: L^2 \to L^\infty_x L^2_t \cap L^{12, 2}_x L^2_t \cap L^{6, 2}_x L^\infty_t \cap L^{24/5, 2}_x L^{8, 2}_t.
		$$
	\end{corollary}
	More generally, admissible reversed Strichartz estimates lie in the quadrilateral with these four vertices.
	\begin{proof} The estimates for $S_{1/2}^H$ and $C_{1/2}^H$ can be proved by the same $T T^*$ argument used in the proof of Theorem \ref{thm_partial_results}, starting from our estimates for $C_1^H P_c$ in Proposition \ref{C1H}. In turn, they imply the estimates for $E_{1/2}^H$.
	\end{proof}
	
	For pointwise decay estimates, we have the following analogue of Theorem \ref{thm:fkgpointwise}.
	\begin{theorem}\lb{thm:PKGcosine} Consider the Klein--Gordon equation in $\R^{3+1}$ with Hamiltonian $H=-\Delta+V$, $V \in L^{3/2, 1}$, such that $H$ has no threshold eigenstates or resonance. Then
		$$
		\bigg\|\frac{\cos(t\sqrt{H+1})P_c}{H+1}f\bigg\|_{L^\infty_x} \les |t|^{-1} \|f\|_{L^1}.
		$$
	\end{theorem}
	
	\begin{proof} The $L^1 \to L^\infty$ decay bound is easy to prove using
		$$
		\frac{\cos(t\sqrt{H+1}) P_c}{H+1}(x, y) = \int_{-\infty}^\infty C_1(r=\tau, t) \big(\partial_\lambda R_{V+}(\lambda^2)P_c\big)^\vee(x, y, \tau) \dd \tau.
		$$
		Indeed, $C_1(\cdot, t)$ is uniformly bounded by $t^{-1}$ and the other factor is integrable in $\tau$.
		%
		%
		%
	\end{proof}
	
		By the same method, we can infer the other main pointwise decay estimate for the perturbed evolution:
		\begin{theorem}\lb{PKGpointwise} Consider the Klein--Gordon equation in $\R^{3+1}$ with Hamiltonian $H=-\Delta+V$, $V \in L^{3/2, 1}$, such that $H$ has no threshold bound states (eigenstates or resonance). Then
			\[
			\|E^{H}_{1+ib}(t)P_{c}f\|_{L^\infty_x}\lesssim |t|^{-1}\|f\|_{L^1}.
			\]
			and 
			\[
			\|E^{H}_{\frac{5}{4}+ib}(t)P_{c}f\|_{L^\infty_x}\lesssim |t|^{-\frac{3}{2}}\|f\|_{L^1}.
			\]
		\end{theorem}
		\begin{observation}
			Unlike the sine propagator bound in Theorem 	  \ref{thm:fkgpointwise}, we could not obtain a bound in the same form since the kernel   \eqref{eq:boundkernel} has no convolution structure.
		\end{observation}


		\subsection{Standard Strichartz estimates}
		
		In this section, we prove the standard Strichartz estimates for the perturbed Klein--Gordon equation.
		\subsubsection{Wave-admissible Strichartz estimates}

		First of all, we discuss the wave-admissible Strichartz estimates  using a
		method inspired by Beals \cite{BE}.
		
		Relying on the complex interpolation instead of frequency dyadic decomposition and
		square functions, we deal with all frequencies simultaneously. This is particularly useful in the case of the perturbed equation, where we lose the scaling invariance.
		
		Our approach is based on two main ingredients: $L^{1}\mapsto L^{\infty}$ pointwise dispersion and the $L^{p}$ boundedness of Mihlin multipliers.

		By Theorem \ref{thm:PKGcosine} and Theorem \ref{thm:fkgpointwise},
		the free sine and cosine evolution of the Klein--Gordon equation
		have the same decay rate as the wave equation. Then all the wave admissible
		Strichartz estimates should hold for the Klein-Gordon equation with
		the homogeneous Sobolev spaces replaced by the regular Sobolev space.
		
		Recall that for the free wave equation in $\mathbb{R}^{3+1}$, 
		\begin{equation*}
		u_{tt}-\Delta u=F,\,\,u(0)=u_{0},\,\,u_{t}(0)=u_{1},\label{eq:614}
		\end{equation*}
		one has the following Strichartz estimates: 
		\begin{equation*}
		\left\Vert u\right\Vert _{L_{t}^{\infty}\dot{H}_{x}^{s}\bigcap L_{t}^{p}L_{x}^{q}}+\left\Vert u_{t}\right\Vert _{L_{t}^{\infty}\dot{H}_{x}^{s-1}}\lesssim\left\Vert u_{0}\right\Vert _{\dot{H}_{x}^{s}}+\left\Vert u_{1}\right\Vert _{\dot{H}_{x}^{s-1}}+\left\Vert F\right\Vert _{L_{t}^{\tilde{p}'}L_{x}^{\tilde{q}'}},\label{eq:615}
		\end{equation*}
		where by the scaling invariance, 
		\begin{equation*}
		\frac{1}{p}+\frac{3}{q}=\frac{3}{2}-s=\frac{1}{\tilde{p}'}+\frac{1}{\tilde{q}'}-2,\,2\leq p,q,\tilde{p},\tilde{q}\leq\infty,\label{eq:616}
		\end{equation*}
		and exponents must be wave-admissible:
		\begin{equation*}
		\frac{2}{p}+\frac{2}{q}\leq1,\,\frac{2}{\tilde{p}}+\frac{2}{\tilde{q}}\leq1.\label{eq:617}
		\end{equation*}
		The endpoint $\left(p,q\right)=\left(2,\infty\right)$ is forbidden.
		So for wave-admissible exponents $\left(p,q\right)$, the pairs $\left(\frac{1}{p},\frac{1}{q}\right)$
		cover triangle with vertices $\left(0,0\right)$, $\left(0,\frac{1}{2}\right)$
		and $\left(\frac{1}{2},0\right)$. 

%
		
		Following the above preparations, we next prove the Strichartz
		estimates for the Klein-Gordon equation with wave-admissible exponents.
		\begin{theorem}
			\label{thm:waveStrichartz}Consider the Klein-Gordon equation in $\mathbb{R}^{3+1}$
			with Hamiltonian $H=-\Delta+V$, $V\in L^{\frac{3}{2},1}$ such that
			$H$ has no threshold eigenstates or resonance. Suppose $u$ solves
			\begin{equation*}
			u_{tt}-\Delta u+u+Vu=0,\,u(0)=u_{0},\,u_{t}(0)=u_{1},\label{eq:623}
			\end{equation*}
			then for $0\leq s<\frac{3}{2}$,
			\begin{equation*}
			\left\Vert P_{c}u\right\Vert _{L_{t}^{\infty}H_{x}^{s}\bigcap L_{t}^{p}L_{x}^{q}}+\left\Vert P_{c}u_{t}\right\Vert _{L_{t}^{\infty}H_{x}^{s-1}}\lesssim\left\Vert u_{0}\right\Vert _{H_{x}^{s}}+\left\Vert u_{1}\right\Vert _{H_{x}^{s-1}}+\left\Vert F\right\Vert _{L_{t}^{\tilde{p}'}L_{x}^{\tilde{q}'}},\label{eq:624}
			\end{equation*}
			where
			\begin{equation*}
			\frac{1}{p}+\frac{3}{q}=\frac{3}{2}-s=\frac{1}{\tilde{p}'}+\frac{1}{\tilde{q}'}-2,\,2\leq p,q,\tilde{p},\tilde{q}\leq\infty,\label{eq:625}
			\end{equation*}
			and exponents are wave-admissible:
			\begin{equation*}
			\frac{2}{p}+\frac{2}{q}\leq1,\,\frac{2}{\tilde{p}}+\frac{2}{\tilde{q}}\leq1.\label{eq:626}
			\end{equation*}
			The endpoint $\left(p,q\right)=\left(2,\infty\right)$ is forbidden.
		\end{theorem}
		
		\begin{proof}
			We first prove the homogeneous Strichartz estimates, i.e., $F=0$.
			
			By Theorem \ref{thm:PKGcosine} and Theorem \ref{thm:fkgpointwise}, as well as Appendix \ref{sec:sob},
			\begin{equation}
			\|C_0^H(t) P_c (-\Delta+1)^{-1}\|_{L^\infty_x} \les \left\Vert C_{1}^{H}\left(t\right)P_{c}f\right\Vert _{L_{x}^{\infty}}\lesssim\left|t\right|^{-1}\left\Vert f\right\Vert _{L_{x}^{1}}.\label{eq:C1decay}
			\end{equation}
			By spectral calculus, we also have 
			\begin{equation*}
			\left\Vert C_{0}^{H}\left(t\right)P_{c}f\right\Vert _{L_{x}^{2}}=\left\Vert \cos\left(t\sqrt{H+1}\right)P_{c}f\right\Vert _{L_{x}^{2}}\lesssim\left\Vert f\right\Vert _{L_{x}^{2}}.\label{eq:627}
			\end{equation*}
			In order to use the complex interpolation method, one needs to consider $\left(-\Delta+1\right)^{i\sigma}$, which is bounded on the Hardy space $\mathcal{H}^{1}$ that contains $L^1$.
			Therefore one has 
			\begin{equation*}
			\left\Vert \cos\left(t\sqrt{H+1}\right)P_c(-\Delta+1)^{-1+i\sigma}f\right\Vert _{L_{x}^{\infty}}\lesssim\left|t\right|^{-1}\left\Vert f\right\Vert _{\mathcal{H}_{x}^{1}}\label{eq:628}
			\end{equation*}
			and 
			\begin{equation*}
			\left\Vert \cos\left(t\sqrt{H+1}\right)P_{c} \left(-\Delta+1\right)^{i\sigma} f\right\Vert _{L_{x}^{2}}\lesssim\left\Vert f\right\Vert _{L_{x}^{2}}.\label{eq:629}
			\end{equation*}
			Then by the complex interpolation, for $0\leq s<1$, 
			\begin{equation*}
			\left\Vert \cos\left(t\sqrt{H+1}\right)P_{c}\left(-\Delta+1\right)^{-s}f\right\Vert _{L_{x}^{\frac{2}{1-s}}}\lesssim\left|t\right|^{-s}\left\Vert f\right\Vert _{L_{x}^{\frac{2}{1+s}}}.\label{eq:630}
			\end{equation*}
			Therefore, by Young's inequality, 
			\begin{equation*}
			\cos\left(t\sqrt{H+1}\right)P_{c}\left(-\Delta+1\right)^{-s}\in\B(L_{t}^{q_{1},r}L_{x}^{\frac{2}{1+s}},L_{t}^{q_{2},r}L_{x}^{\frac{2}{1-s}}),\label{eq:631}
			\end{equation*}
			where $\frac{1}{q_{2}}=\frac{1}{q_{1}}+s-1$, with the standard modification
			at the endpoints, i.e., when $q_{1}=1$ or $q_{2}=\infty$, then one
			needs to replace $L_{t}^{q_{1},r}$ by $L_{t}^{q_{1},1}$ and replace
			$L_{t}^{q_{2},r}$ by $L_{t}^{q_{2},\infty}$. Applying Lemmas 	\ref{lem:sob1} and \ref{lem:sob} to replace $\left(-\Delta+1\right)^{-s}$ by $\left(H+1\right)^{-s}$, and then making the spaces dual
			to each other, one has
			\begin{equation*}
			\cos\left(t\sqrt{H+1}\right)P_{c}\left(H+1\right)^{-s}\in\mathcal{B}\left(L_{t}^{\frac{2}{1-s},2}L_{x}^{\frac{2}{1+s}},L_{t}^{\frac{2}{s},2}L_{x}^{\frac{2}{1-s}}\right).\label{eq:632}
			\end{equation*}
			By a $TT^{*}$ argument, we can conclude that for $0\leq s<1$
			\begin{equation}
			\frac{\sin\left(t\sqrt{H+1}\right)}{\left(H+1\right)^{s/2}}P_{c},\,\frac{\cos\left(t\sqrt{H+1}\right)}{\left(H+1\right)^{s/2}}P_{c}\in\mathcal{B}\left(L_{x}^{2},L_{t}^{\frac{2}{s},2}L_{x}^{\frac{2}{1-s}}\right)\lb{eq:633}
			\end{equation}
			which with Lemma \ref{lem:sob} together gives the sharp wave-admissible
			homogeneous Strichartz estimates. 
			
			By Lemma \ref{lem:sob}, for $\left|s\right|<\frac{3}{2}$, the
			standard Sobolev spaces are equivalent to those Sobolev spaces defined
			by fractional integration with respect to $H$. Then we have 
			\begin{equation*}
			\| P_{c}u \| _{L_{t}^{\infty}H_{x}^{s}} + \| P_{c}u_t \| _{L_{t}^{\infty}H_{x}^{s-1}}\lesssim\left\Vert u_{0}\right\Vert _{H_{x}^{s}}+\left\Vert u_{1}\right\Vert _{H_{x}^{s-1}}\label{eq:634}
			\end{equation*}
			for $-\frac{1}{2}<s<\frac{3}{2}$. By Sobolev embedding, the above
			energy estimate implies the homogeneous Strichartz estimates along
			the line segment $(p,q)=\left(\infty,\frac{6}{3-2s}\right)$ for $0\leq s<\frac{3}{2}.$
			Finally, interpolating with the sharp homogeneous estimates \eqref{eq:633},
			we obtain the full range of homogeneous Strichartz estimates.
			
			For the inhomogenous estimates, the sharp estimates \eqref{eq:633}
			imply the sharp dual estimates. Again using the energy estimates,
			one can obtain the bound for $(\tilde{p}',\tilde{q}')=\left(1,\frac{6}{3+2s}\right)$
			when $0\leq s<\frac{3}{2}$. Then interpolating between these two cases as
			the homogeneous estimates, we obtain the full range inhomogeneous
			Strichartz estimates.
		\end{proof}
		\subsubsection{The endpoint Strichartz estimates}
		
		Now we discuss the endpoint Strichartz estimates.
		Keel--Tao \cite{KeTa} provides a general framework, based on real interpolation,
		to prove endpoint Strichartz estimates for operators satisfying
		certain conditions.  
		
		To establish the endpoint estimates here, we have to invoke Keel-Tao's approach.  The Littlewood-Paley theory for perturbed Hamiltonians is needed here. We recall results from Beceanu-Goldberg \cite{becgol1}.
		
		Firstly, we define the Littlewood-Paley square function for the perturbed Hamiltonian
		as
		\[
		\left[S_{H}f\right]\left(x\right)=\left(\sum_{k\in\mathbb{Z}}\left|\left[\chi\left(2^{-k}\sqrt{H}\right)f\right]\left(x\right)\right|^{2}\right)^{\frac{1}{2}}.
		\]
		where $\chi \in C_{c}^{\infty}(\mathbb{R})$  is a fixed smooth cutoff function such that $\sum_{k \in \mathbb{Z}} \chi\left(2^{-k} t\right)=\chi_{(0, \infty)}(t)$ and
		$\operatorname{supp} \chi \subset[2 / 3,3].$
		\begin{theorem}[Theorem 4.5 \cite{becgol1}] \label{thm:BG} Suppose $V\in L^{\frac{3}{2},1}$ and $H=-\Delta+V$ has no eigenvalue or
			resonance at zero, and no positive eigenvalues. Then for each $p\in\left(1,\infty\right)$,
			one has
			\[
			\left\Vert S_{H}f\right\Vert _{L^{p}}\lesssim\left\Vert f\right\Vert _{L^{p}}\lesssim\left\Vert S_{H}f\right\Vert _{L^{p}}.
			\]
		\end{theorem}
		
		Motivated by the formulation in D'Ancona--Fanelli \cite{DaF}, Strichartz estimates
		for Klein-Gordon equations are formulated as follows: 
		\begin{theorem}
			\label{thm:KGStrichartz} Consider the Klein-Gordon  in $\mathbb{R}^{3+1}$
			with Hamiltonian $H=-\Delta+V$, $V\in L^{\frac{3}{2},1}$, such that
			$H$ has no threshold eigenstates or resonance. Then for 
			\begin{equation}
			\frac{2}{p}+\frac{3}{q}=\frac{3}{2},\,p\geq2\label{eq:661}
			\end{equation}
			one has the homogeneous Strichartz estimates
			\begin{equation}
			\left\Vert E_{\frac{1}{2}\left(\frac{1}{p}+\frac{1}{2}-\frac{1}{q}\right)}^{H}\left(t\right)P_{c}f\right\Vert _{L_{t}^{p}L_{x}^{q}}\lesssim\left\Vert f\right\Vert _{L_{x}^{2}}\label{eq:662}
			\end{equation}
			and the inhomogeneous Strichartz estimates
			\begin{equation}
			\left\Vert \int_{-\infty}^{t}E_{\frac{1}{2}\left(\frac{1}{p}+\frac{1}{2}-\frac{1}{q}\right)}^{H}\left(t\right)\left(E_{\frac{1}{2}\left(\frac{1}{p}+\frac{1}{2}-\frac{1}{q}\right)}^{H}\left(s\right)\right)^{*}P_{c}F(s)\,ds\right\Vert _{L_{x}^{q}L_{x}^{p}}\lesssim\left\Vert F\right\Vert _{L_{t}^{\tilde{p}'}L_{x}^{\tilde{q}'}}\label{eq:663}
			\end{equation}
			where $\left(\tilde{p},\tilde{q}\right)$ satisfies \eqref{eq:661}
			and $\frac{1}{\tilde{p}}+\frac{1}{\tilde{p}'}=1$,  $\frac{1}{\tilde{q}}+\frac{1}{\tilde{q}'}=1$.
		\end{theorem}
		
		When $p=\infty$, then $q=2$, the estimate is trivial. We attempt
		to understand the endpoint estimates with $p=2$. Note that in the
		above Strichartz estimates, if we do not apply the Littlewood-Paley
		decomposition, the evolution operator varies with respect to admissible
		pairs.
		
		We will rely on Theorem \ref{thm:PKGcosine} and Theorem
		\ref{KGpointDecay} 
		\begin{equation}
		\left\Vert E_{\frac{5}{4}}^{H}\left(\cdot,t\right)P_{c}f\right\Vert _{L_{x}^{\infty}}\lesssim\left|t\right|^{-\frac{3}{2}}\left\Vert f\right\Vert _{L_{x}^{1}}\label{eq:pointdecay1}
		\end{equation}
		and 
		\begin{equation}
		\left\Vert E_{0}^{H}\left(\cdot,t\right)P_{c}f\right\Vert _{L_{x}^{2}}\lesssim\left\Vert f\right\Vert _{L_{x}^{2}}.\label{eq:L2}
		\end{equation}
		In the following, we will prove the endpoint  $p=2$ and $q=6$.
		The full range in Theorem \ref{thm:KGStrichartz} follows by interpolating
		$p=\infty,\,q=2$ case with the endpoint estimates.
		
		\begin{proof}[Proof of Theorem \ref{thm:KGStrichartz}]
			By our definition, one has
			\[
			\left|\chi\left(2^{-k}\sqrt{H}\right)e^{it\sqrt{H+1}}P_{c}f\right|_{L^{\infty}}\sim\left|\left\langle 2^{\frac{5}{2}k}\right\rangle \chi\left(2^{-k}\sqrt{H}\right)\frac{e^{it\sqrt{H+1}}}{\left(\sqrt{H+1}\right)^{\frac{5}{2}}}P_{c}f\right|_{L^{\infty}}.
			\]
			By the pointwise decay \eqref{eq:pointdecay1}, it follows
			\begin{align*}
			\left|\left\langle 2^{\frac{5}{2}k}\right\rangle \chi\left(2^{-k}\sqrt{H}\right)\frac{e^{it\sqrt{H+1}}}{\left(\sqrt{H+1}\right)^{\frac{5}{2}}}P_{c}f\right|_{L^{\infty}} & \lesssim\left\langle 2^{\frac{5}{2}k}\right\rangle \left|\chi\left(2^{-k}\sqrt{H}\right)\frac{e^{it\sqrt{H+1}}}{\left(\sqrt{H+1}\right)^{\frac{5}{2}}}P_{c}f\right|_{L^{\infty}}\\
			& \lesssim\left\langle 2^{\frac{5}{2}k}\right\rangle \frac{1}{t^{\frac{3}{2}}}\left\Vert \chi\left(2^{-k}\sqrt{H}\right)P_{c}f\right\Vert _{L^{1}}.
			\end{align*}
			We also know from \eqref{eq:L2}
			\[
			\left\Vert \chi\left(2^{-k}\sqrt{H}\right)e^{it\sqrt{H+1}}P_{c}f\right\Vert _{L_{x}^{2}}\lesssim\left\Vert \chi\left(2^{-k}\sqrt{H}\right)P_{c}f\right\Vert _{L_{x}^{2}}.
			\]
			Applying the abstract interpolation theorem due to Keel--Tao \cite{KeTa}, we obtain
			that
			\[
			\int\left\Vert \chi\left(2^{-k}\sqrt{H}\right)e^{it\sqrt{H+1}}P_{c}f\right\Vert _{L_{x}^{6}}^{2}\,dt\lesssim\left\langle 2^{\frac{5}{6}k}\right\rangle \left\Vert \chi\left(2^{-k}\sqrt{H}\right)P_{c}f\right\Vert _{L_{x}^{2}}^{2}.
			\]
			Therefore
			\[
			\int\left\Vert \chi\left(2^{-k}\sqrt{H}\right)\frac{e^{it\sqrt{H+1}}}{\left(\sqrt{H+1}\right)^{\frac{5}{6}}}P_{c}f\right\Vert _{L_{x}^{6}}^{2}\,dt\lesssim\left\Vert \chi\left(2^{-k}\sqrt{H}\right)P_{c}f\right\Vert _{L_{x}^{2}}^{2}.
			\]
			Then applying Theorem \ref{thm:BG} and Minkowski inequality, one has
			\begin{align*}
			\left\Vert \frac{e^{it\sqrt{H+1}}}{\left(\sqrt{H+1}\right)^{\frac{5}{6}}}P_{c}f\right\Vert _{L_{x}^{6}}^{2} & =\left\Vert \sum_{k\in\mathbb{Z}}\chi\left(2^{-k}\sqrt{H}\right)\frac{e^{it\sqrt{H+1}}}{\left(\sqrt{H+1}\right)^{\frac{5}{6}}}P_{c}f\right\Vert _{L_{x}^{6}}^{2}\\
			& \lesssim\left\Vert \left(\sum_{k\in\mathbb{Z}}\left|\chi\left(2^{-k}\sqrt{H}\right)\frac{e^{it\sqrt{H+1}}}{\left(\sqrt{H+1}\right)^{\frac{5}{6}}}P_{c}f\right|^{2}\right)^{\frac{1}{2}}\right\Vert _{L_{x}^{6}}^{2}\\
			& \lesssim\sum_{k\in\mathbb{Z}}\left\Vert \chi\left(2^{-k}\sqrt{H}\right)\frac{e^{it\sqrt{H+1}}}{\left(\sqrt{H+1}\right)^{\frac{5}{6}}}P_{c}f\right\Vert _{L_{x}^{6}}^{2}.
			\end{align*}
			Hence it follows that
			\begin{align*}
			\int\left\Vert \frac{e^{it\sqrt{H+1}}}{\left(\sqrt{H+1}\right)^{\frac{5}{6}}}P_{c}f\right\Vert _{L_{x}^{6}}^{2}dt & \lesssim\sum_{k\in\mathbb{Z}}\int\left\Vert \chi\left(2^{-k}\sqrt{H}\right)\frac{e^{it\sqrt{H+1}}}{\left(\sqrt{H+1}\right)^{\frac{5}{6}}}P_{c}f\right\Vert _{L_{x}^{6}}^{2}dt\\
			& \lesssim\sum_{k\in\mathbb{Z}}\left\Vert \chi\left(2^{-k}\sqrt{H}\right)P_{c}f\right\Vert _{L_{x}^{2}}^{2}\\
			& \lesssim\left\Vert P_{c}f\right\Vert _{L_{x}^{2}}^{2}.
			\end{align*}Inhomogeneous estimates follow in the same manner.
		\end{proof}

	\section{Nonlinear applications}\label{sec:nonlinear}
	In this section, we use some linear estimates from previous sections to establish some well-posedness results for the associated semilinear equations. 	To illustrate the ideas, we focus on the quintic case.
	
	
	When bootstrapping for the wave equation with $\dot H^1$ initial data, the solution had $1/2$ powers of decay, the nonlinearity had $5/2$ powers of decay, the wave equation propagator had $2$ powers of decay, so the contribution of the nonlinearity had $5/2+2-4=1/2$ powers of decay and we could close the loop.
	
	In the Klein--Gordon case, we have to do this for both the wave and the Bessel part. If the solution has $\alpha$ powers of decay, the nonlinearity needs $\alpha+2$ and $\alpha+5/2$ powers of decay, in order to get back $(\alpha+2)+2-4=\alpha$ (for the wave kernel), respectively $(\alpha+5/2)+3/2-4=\alpha$ (for the Bessel part) decay after bootstrapping.
	
	Hence the nonlinear potential $u^4$ needs both $2$ and $5/2$ powers of decay. The first requirement can be satisfied in the same way as for the wave equation. For the second requirement, we need to have two of the $u$ factors within the nonlinear potential in Strichartz norms with $1/4$ power extra decay, which is the most possible.
	
	More precisely, the nonlinear potential $u^4$ ought to be in $L^{3/2, q}_x L^\infty_t$ for the wave part and in e.g.\;$L^{4/3, q}_x L^{4, \infty}_t$ or $L^{6/5, q}_x L^\infty_t$ for the Bessel part.
	
	Hence $u$ itself has to be in $L^6_x L^\infty_t$ and in $L^{16/3}_x L^{16}_t$ or $L^{24/5}_x L^\infty_t$ or some intermediate space.
	
	These latter powers are half-way between the wave-admissible and the endpoint Klein--Gordon-admissible powers.
	
	We thus get the following well-posedness for the quintic equation.
	\begin{proposition}\lb{prop_nonlinear} Consider a potential $V \in L^{3/2, 1}(\R^3)$ such that $H-\Delta+V$ has no negative eigenvalues and no threshold eigenfunctions or resonance. The quintic Klein--Gordon equation in $\R^{3+1}$
		$$
		u_{tt}-\Delta u + Vu + u \pm u^5 = F,\ u(0)=v,\ u_t(0)=w
		$$
		is globally well-posed if the norm of the initial data $\|(v, w)\|_{H^1 \times L^2}$ is sufficiently small, if $F$ is also small in the norm of
		$
		L^{6/5, 2}_x L^\infty_t \cap L^{16/15, 2}_x L^{16/5, 2}_t \cap L^{12/11, 2}_x L^{4, 1}_t \cap L^{48/41, 2}_x L^{16/3, 2}_t.
		$
		The solution $u$ is in $L^{6, 2}_x L^\infty_t \cap L^{16/3, 2}_x L^{16, 2}_t$ and
		$$
		\|u\|_{L^{6, 2}_x L^\infty_t \cap L^{16/3, 2}_x L^{16, 2}_t} \les \|(v, w)\|_{H^1 \times L^2} + \|F\|_{L^{6/5, 2}_x L^\infty_t \cap L^{16/15, 2}_x L^{16/5, 2}_t \cap L^{12/11, 2}_x L^{4, 1}_t \cap L^{48/41, 2}_x L^{16/3, 2}_t}.
		$$
		
		If in addition $F=0$ or $F$ is small in several more norms, then $u \in L^\infty_x L^2_t \cap L^{12, 2}_x L^2_t \cap L^{24/5, 2}_x L^{8, 2}_t$ as well.
	\end{proposition}
	The equation is still locally well-posed in these norms if the initial data or $F$ are large.
	
	Concerning the inhomogenous term, note that $L^{4/3} \cdot L^6 = L^{12/11}$ and $L^{3/2} \cdot L^{16/3} = L^{48/41}$.
	\begin{proof} We focus on the homogeneous case first. The solution will be obtained as the limit of a sequence of approximations. Linearize the equation and recursively construct the sequence $(u^n)$ such that $u_0=0$ and
		$$
		(u_n)_{tt}-\Delta u_n + V u_n + u_n \pm (u_{n-1})^5 = 0,\ u_n(0)=v,\ (u_n)_t(0)=w.
		$$
		The nonlinear potential $(u^{n-1})^4$ being small in the $L^{3/2, 1}_x L^\infty_t \cap L^{4/3, 1}_x L^{4, 1}_t$ norm suffices to close the loop.
		
		More precisely, the solution is given by Duhamel's formula: 
		\[
		u_{n}=S_{1/2}^H(t)u_{n}\left(0\right)+C^H_{0}\left(t\right)\left(u_{n}\right)_{t}\left(0\right)+\int_{0}^{t}S^H_{1/2}\left(t-s\right)u_{n-1}^{5}\,ds.
		\]
		
		By our homogeneous Strichartz estimates
		$$
		\|S_{1/2}^H(t)v+C^H_{0}\left(t\right)w\|_{L_{x}^{6,2}L_{t}^{\infty}\cap L_{x}^{\frac{16}{3},2}L_{t}^{16,2}} \les \|(v, w)\|_{H^1 \times L^2}.
		$$
		
		As in Section \ref{subsec:sinpro}, we split the sine propagator into wave and
		Bessel parts:
		\[
		S^H_{1/2}(t-s)=S^H_{W}(t-s)+S^H_{B}(t-s).
		\]
		Suppose $u_{n-1}\in L_{x}^{6,2}L_{t}^{\infty}\cap L_{x}^{\frac{16}{3},2}L_{t}^{16,2}$.
		Then $u_{n-1}^{4} \in L_{x}^{\frac{3}{2},2}L_{t}^{\infty}\cap L_{x}^{\frac{4}{3},2}L_{t}^{4,2}$ and
		\begin{equation}
		\|u_{n-1}^{4}\|_{L_{x}^{\frac{3}{2},2}L_{t}^{\infty}\cap L_{x}^{\frac{4}{3},2}L_{t}^{4,2}} \les \|u_{n-1}\|_{L_{x}^{6,2}L_{t}^{\infty}\cap L_{x}^{\frac{16}{3},2}L_{t}^{16,2}}^4.\label{eq:estf2}
		\end{equation}
		We try to estimate the $L_{x}^{6,2}L_{t}^{\infty}\cap L_{x}^{\frac{16}{3},2}L_{t}^{16,2}$
		norm of $u_{n}$.
		
		For the wave part, 
		by the Strichartz estimates implied by Theorem \ref{SBH} and H\"older's inequality, as in the quintic wave case, one has
		\[
		\bigg\|\int_{0}^{t}S^H_{W}\left(t-s\right)u_{n-1}^{5}(s)\,ds\bigg\| _{L_{x}^{6,2}L_{t}^{\infty}} \les \|u_{n-1}^5\|_{L_{x}^{6/5,2}L_{t}^{\infty}} \les \|u_{n-1}\|_{L^{6, 2}_x L^\infty_t}^5,
		\]
		
		We next estimate the $L_{x}^{6,2}L_{t}^{\infty}$
		norm of the Bessel part: as in Section \ref{subsec:sinpro},
		$$
		\bigg\| \int_{0}^{t}S^H_{B}\left(t-s\right)u_{n-1}^5(s)\,ds\bigg\|_{L_{x}^{6,2}L_{t}^{\infty}} \les \|u_{n-1}^{5}\| _{L_{x}^{12/11,2}L_{t}^{4,2}} \leq\left\Vert u_{n-1}\right\Vert _{L_{x}^{6,2}L_{t}^{\infty}}\left\Vert u_{n-1}^4\right\Vert _{L_{x}^{4/3,2}L_{t}^{4,2}}
		$$
		by H\"older's inequality.
		
		Next, we estimate the $L_{x}^{\frac{16}{3},2}L_{t}^{16,2}$ norm.
		For the wave part, one has
		\[
		\left\Vert \int_{0}^{t}S^H_{W}\left(t-s\right)\left(u_{n-1}\right)^{4}\left(u_{n-1}\right)(s)\,ds\right\Vert _{L_{x}^{\frac{16}{3},2}L_{t}^{16,2}}\lesssim\left\Vert \left(u_{n-1}\right)^{4}\left(u_{n-1}\right)\right\Vert _{L_{x}^{\frac{48}{11},2}L_{t}^{16,2}}.
		\]
		Again by H\"older's inequality, 
		\[
		\left\Vert \left(u_{n-1}\right)^{4}\left(u_{n-1}\right)\right\Vert _{L_{x}^{\frac{48}{11},2}L_{t}^{16,2}}\lesssim\left\Vert u_{n-1}\right\Vert _{L_{x}^{\frac{16}{3},2}L_{t}^{16,2}}\left\Vert u_{n-1}^{4}\right\Vert _{L_{x}^{\frac{3}{2},2}L_{t}^{\infty}}
		\]
		which can be bounded via \eqref{eq:estf2}. 
		
		Finally, we apply the $L_{x}^{\frac{16}{3},2}L_{t}^{16,2}$ norm to
		the Bessel part, we obtain
		\begin{align*}
		\left\Vert \int_{0}^{t}S^H_{B}\left(t-s\right)\left(u_{n-1}\right)^{4}\left(u_{n-1}\right)(s)\,ds\right\Vert _{L_{x}^{\frac{16}{3},2}L_{t}^{16,2}} & \lesssim\left\Vert \left(u_{n-1}\right)^{4}\left(u_{n-1}\right)\right\Vert _{L_{x}^{\frac{16}{15},2}L_{t}^{\frac{16}{3},2}}\\
		& \lesssim\left\Vert u_{n-1}\right\Vert _{L_{x}^{\frac{16}{3},2}L_{t}^{16,2}}\left\Vert u_{n-1}^{4}\right\Vert _{L_{x}^{\frac{4}{3},2}L_{t}^{4,2}},
		\end{align*}
		which can also be bounded by \eqref{eq:estf2}.
		
		From the above discussion, we know that if we set 
		\[
		u_{n}=S^H_{\frac{1}{2}}(t)u_{n}\left(0\right)+C^H_{0}\left(t\right)\left(u_{n}\right)_{t}\left(0\right)+\int_{0}^{t}S^H_{\frac{1}{2}}\left(t-s\right)\left(u_{n-1}\right)^{4}\left(u_{n-1}\right)(s)\,ds
		\]
		then 
		\begin{align*}
		\left\Vert u_{n}\right\Vert _{L_{x}^{6,2}L_{t}^{\infty}\bigcap L_{x}^{\frac{16}{3},2}L_{t}^{16,2}} & \lesssim\left\Vert (v, w) \right\Vert _{H^{1}\times L^{2}} +\left\Vert u_{n-1}\right\Vert _{L_{x}^{6,2}L_{t}^{\infty}\cap L_{x}^{\frac{16}{3},2}L_{t}^{16,2}}^{5}.
		\end{align*}
		Likewise
		\begin{align*}
		\left\Vert u_{n}-u_{n-1}\right\Vert _{L_{x}^{6,2}L_{t}^{\infty}\bigcap L_{x}^{\frac{16}{3},2}L_{t}^{16,2}} & \lesssim \left\Vert u_{n-1}-u_{n-2}\right\Vert_{L_{x}^{6,2}L_{t}^{\infty}\cap L_{x}^{\frac{16}{3},2}L_{t}^{16,2}} \big(\left\Vert u_{n-1}\right\Vert _{L_{x}^{6,2}L_{t}^{\infty}\cap L_{x}^{\frac{16}{3},2}L_{t}^{16,2}}^{4} \\
		&+\left\Vert u_{n-2}\right\Vert_{L_{x}^{6,2}L_{t}^{\infty}\cap L_{x}^{\frac{16}{3},2}L_{t}^{16,2}}^{4}\big).
		\end{align*}
		Therefore, the sequence $(u_n)_n$ is contractive if the initial data have sufficiently small norm in $H^1 \times L^2$ and $L_{x}^{6,2}L_{t}^{\infty}\cap L_{x}^{\frac{16}{3},2}L_{t}^{16,2}$
		is enough to close the iteration loop.
		
		The contribution of the inhomogeneous term $F$ is treated in the same manner. The extra norms are also obtained by iteration, if $F$ is finite in the appropriate norms.
	\end{proof}

		\begin{observation}
			For higher order powers, we refer to Beceanu-Goldberg \cite{becgol} for the wave equation. After some modifications adapted to the Bessel parts as in the analysis above, we can also obtain some well-posedness results for higher order Klein-Gordon equations.
		\end{observation}

	\appendix
	
	\section{Bessel and Hankel functions}\label{sec:bessel}
	In this section we state some known facts about the Bessel functions used in this paper. A good reference concerning Bessel or Hankel functions, or special functions in general, is for example Andrews--Askey--Roy \cite{aar}.\\
	
	\noindent\emph{Bessel functions of the first kind} $J_n$, of integer order $n \in \Z$, can be represented by the integral
	\be\lb{bessel_int}
	J_n(x) = \frac 1 {2\pi} \int_{-\pi}^\pi e^{i(nt + x\sin t)} \dd t.
	\ee
	This integral representation implies that $|J_n(x)| \leq 1$. Also, one has the following uniform decay bound:
	
	\begin{lemma}\lb{universal_bessel} $J_n(x) \les \frac {\langle n \rangle^{1/2}} {x^{1/2}}$.
	\end{lemma}
	\begin{proof}
		This integral has two points of stationary phase, namely $x = \pm \pi/2$. We integrate by parts outside two intervals of size $2\epsilon$, one around each point. By the method of stationary phase, we get
		$$\begin{aligned}
		|J_n(x)| &\les \epsilon + \frac{e^{int}}{x \cos t} e^{i(x\sin t)} \bigg|_{t=-\pi}^{-\pi/2-\epsilon} + \bigg|_{-\pi/2+\epsilon}^{\pi/2-\epsilon} + \bigg|_{\pi/2+\epsilon}^{\pi} - \\
		&\bigg(\int_{-\pi}^{-\pi/2-\epsilon} + \int_{-\pi/2+\epsilon}^{\pi/2-\epsilon} + \int_{\pi/2+\epsilon}^{\pi}\bigg) \bigg(\frac {e^{int}} {x \cos t}\bigg)' \cos(nt + x\sin t) \dd t \\
		&\les \frac {\langle n \rangle} {x\epsilon} + \epsilon \les \frac {\langle n \rangle^{1/2}} {x^{1/2}},
		\end{aligned}$$
		for an appropriate choice of $\epsilon=\frac {\langle n \rangle^{1/2}} {|x|^{1/2}}$.
	\end{proof}
	
	\noindent \emph{Recurrence relations}, see \cite{aar}:
	$$
	\alpha J_\alpha(x) +xJ'_\alpha(x) = x J_{\alpha-1}(x),\ 
	-\alpha J_\alpha(x)+xJ_\alpha'(x) = -xJ_{\alpha+1}(x).
	$$
	Hence $J_0'=-J_1$ and
	$$
	\int_\sigma^\infty J_1(z) \dd z = J_0(\sigma).
	$$
	Thus, the integral is also never more than $1$ in absolute value:
	$$
	\sup_{r \geq 0} \bigg|\int_r^\infty J_1(z) \dd z\bigg| < \infty.
	$$
	In particular
	\be\lb{intj}
	\int_0^\infty J_1(s) \dd s = J_0(0) = 1.
	\ee
	Other useful formulas are
	$$
	J_1'(z)=J_0(z)-J_1(z)/z,\ (J_1(z)/z)'=-J_2(z)/z.
	$$
	
	\noindent \emph{Hankel functions} or Bessel functions of the third kind are defined as
	$$
	H_\alpha^\pm = J_\alpha \pm i Y_\alpha.
	$$
	
	\noindent\emph{Asymptotic behavior:} For $|z|<<1$,
	\be\lb{asymp}
	J_1(z) \sim \frac z 2,\ Y_1(z) \sim - \frac 2 {\pi z},\ H_1^+(z) \sim - \frac {2i} {\pi z}.
	\ee
	For $|z|>>1$,
	$$
	H_1^+(z) \sim \sqrt{\frac 2 {\pi z}} e^{i(z - 3\pi/4)}.
	$$
	
	\section{Sharp Agmon estimates}\label{sec:agmon}
	
	In this appendix, we provide a proof of the sharp Agmon estimates
	for the sake of completeness. 
	\begin{theorem}
		\label{thm:sharpAgmon} Suppose $\psi$ is an eigenfunction associated to the negative eigenvalue $-\lambda^{2}$ of the Schr\"{o}dinger
		operator $-\Delta+V$, where $V\in L^{\frac{3}{2},1}$:
		$$
		(-\Delta+V) \psi = -\lambda^2 \psi.
		$$
		Then 
		\begin{equation*}
		|\psi(x)|\les\left\langle x\right\rangle ^{-1}e^{-\lambda\left|x\right|}.\label{eq:sharpAgmon}
		\end{equation*}
	\end{theorem}
	One can also use the Kato norm instead of the $L^{3/2, 1}$ norm.
	\begin{proof}
		One can write 
		\begin{equation}
		\psi=-R_{0}\left(-\lambda^{2}\right)V\psi,\label{eq:resol}
		\end{equation}
		where $R_{0}$ is the free resolvent. By the standard method \cite{Agmon},
		one can show that $\psi\in L^{\infty}$. For smooth potentials, one
		can find a proof via the wave equation in \cite{C1}. 
		
		Write 
		\begin{equation}
		V=V_{1}+V_{2}\label{eq:decV}
		\end{equation}
		where $V_{1}$ is bounded with compact support and the $L^{\frac{3}{2},1}$
		norm of $V_{2}$ is small. From \eqref{eq:resol} and \eqref{eq:decV},
		\begin{equation}
		\psi+R_{0}\left(-\lambda^{2}\right)V_{2}\psi=-R_{0}\left(-\lambda^{2}\right)V_{1}\psi.\label{eq:decoV}
		\end{equation}
		Consider the space 
		\[
		e^{-\lambda\left|x\right|}\left\langle x\right\rangle ^{-1}L_{x}^{\infty}:=\left\{ f|\,\,e^{\lambda\left|x\right|}\left\langle x\right\rangle f\in L_{x}^{\infty}\right\} .
		\]
		
		Clearly
		\[
		-R_{0}\left(-\lambda^{2}\right)V_{1}\psi\in e^{-\lambda\left|x\right|}\left\langle x\right\rangle ^{-1}L_{x}^{\infty},
		\]
		since $V_{1}$ is smooth and compactly supported. We only have to
		show that $\left(I+R_{0}\left(-\lambda^{2}\right)V_{2}\right)$ is
		invertible in $e^{-\lambda\left|x\right|}\left\langle x\right\rangle ^{-1}L_{x}^{\infty}$.
		The invertiblity follows from the smallness of $V_{2}$ which implies
		the smallness of the operator norm of $R_{0}\left(-\lambda^{2}\right)V_{2}$
		in $\B(e^{-\lambda\left|x\right|}\left\langle x\right\rangle ^{-1}L_{x}^{\infty})$.
		
		To be more precise, with the explicit expression for $R_{0}\left(-\lambda^{2}\right)$,
		it suffices to show 
		\[
		\int e^{-\lambda\left|x-y\right|}\frac{1}{\left|x-y\right|}V_{2}\left(y\right)e^{-\lambda\left|y\right|}\frac{1}{\left\langle y\right\rangle }\,dy\leq Ce^{-\lambda\left|x\right|}\left\langle x\right\rangle ^{-1}
		\]
		for some small $C$ independent of $x$. This reduces to showing 
		\[
		\int\frac{1}{\left|x-y\right|}V_{2}\left(y\right)\frac{1}{\left\langle y\right\rangle }\,dy\leq C\left\langle x\right\rangle ^{-1}.
		\]
		The Japanese brackets are unimportant since originally, we know our
		function is in $L^{\infty}$. 
		
		If $\left|y\right|>\frac{1}{2}\left|x\right|$, then $\frac{1}{\left\langle y\right\rangle }\lesssim\frac{1}{\left\langle x\right\rangle }$
		and hence 
		\begin{align}
		\int\frac{1}{\left|x-y\right|}V_{2}\left(y\right)\frac{1}{\left\langle y\right\rangle }\,dy & \lesssim\frac{1}{\left\langle x\right\rangle }\sup_{x\in\mathbb{R}^{3}}\int\frac{1}{\left|x-y\right|}\left|V_{2}\left(y\right)\right|\,dy\nonumber \\
		& \leq\left\Vert V_{2}\right\Vert _{L^{\frac{3}{2},1}}\left\langle x\right\rangle ^{-1}.\label{eq:reg1}
		\end{align}
		When $\left|y\right|\leq\frac{1}{2}\left|x\right|$, then $\frac{1}{\left|x-y\right|}\leq2\frac{1}{\left|x\right|}$
		and therefore,
		\begin{align}
		\int\frac{1}{\left|x-y\right|}V_{2}\left(y\right)\frac{1}{\left|y\right|}\,dy\lesssim & \frac{1}{\left|x\right|}\int\frac{1}{\left|y\right|}\left|V_{2}\left(y\right)\right|\,dy\nonumber \\
		\leq & \frac{1}{\left|x\right|}\sup_{x\in\mathbb{R}^{3}}\int\frac{1}{\left|x-y\right|}\left|V_{2}\left(y\right)\right|\,dy\nonumber \\
		& \leq\left\Vert V_{2}\right\Vert _{L^{\frac{3}{2},1}}\left|x\right|^{-1}.\label{eq:reg2}
		\end{align}
		By \eqref{eq:reg1} and \eqref{eq:reg2}, we can conclude that the operator
		norm of $R_{0}\left(-\lambda^{2}\right)V_{2}$ in $e^{-\lambda\left|x\right|}\left\langle x\right\rangle ^{-1}L_{x}^{\infty}$
		is small provided $\left\Vert V_{2}\right\Vert _{L^{\frac{3}{2},1}}$
		is sufficiently small and $\left(I+R_{0}\left(-\lambda^{2}\right)V_{2}\right)$
		is invertible in $e^{-\lambda\left|x\right|}\left\langle x\right\rangle ^{-1}L_{x}^{\infty}$.
		From \eqref{eq:decoV},
		\[
		\psi=\left(I+R_{0}\left(-\lambda^{2}\right)V_{2}\right)^{-1}\left(-R_{0}\left(-\lambda^{2}\right)V_{1}\psi\right),
		\]
		and hence
		\[
		\left|\psi\right|\lesssim\left\langle x\right\rangle ^{-1}e^{-\lambda\left|x\right|}
		\]
		as we claimed.
	\end{proof}
	\section{Comparison of Sobolev spaces}\label{sec:sob}
	In this appendix, we show that fractional integration with respect to powers of $H$ behaves similarly to that with respect to powers of the Laplacian.
	
	\begin{lemma}	\label{lem:sob1}
		\label{lem:compare}For $0\leq s\leq1$, 
		\begin{equation*}
		\left(-\Delta+1\right)^{-s}\left(H+1\right)^{s}P_{c},\:P_{c}\left(H+1\right)^{-s}\left(-\Delta+1\right)^{s}\in\mathcal{B}\left(L^{p}\right)\label{eq:61}
		\end{equation*}
		when $\frac{3}{3-2s}<p<\infty$, and 
		
		\begin{equation*}
		\left(-\Delta+1\right)^{s}\left(H+1\right)^{-s}P_{c},\:P_{c}\left(H+1\right)^{s}\left(-\Delta+1\right)^{-s}\in\mathcal{B}\left(L^{p}\right)\label{eq:62}
		\end{equation*}
		when $1<p<\frac{3}{2s}$.
		
		Moreover, 
		\begin{equation*}
		\left(-\Delta+1\right)^{-1}\left(H+1\right)P_{c},\,P_{c}\left(H+1\right)^{-1}\left(-\Delta+1\right)\in\mathcal{B}\left(L^{\infty}\right),\label{eq:63}
		\end{equation*}
		\begin{equation*}
		\left(-\Delta+1\right)\left(H+1\right)^{-1}P_{c},\:P_{c}\left(H+1\right)\left(-\Delta+1\right)^{-1}\in\mathcal{B}\left(L^{1}\right).\label{eq:64}
		\end{equation*}
	\end{lemma}
The proof follows that of Lemma 1.14 in \cite{hong}.	
	\begin{proof}
		Define 
		\begin{equation*}
		R_{0}\left(z\right)=\left(-\Delta-\lambda\right)^{-1}.\label{eq:65}
		\end{equation*}
		Then $R_{0}\left(z\right)$ is analytic on $\mathbb{C}\backslash[0,\infty)$,
		with the explicit kernel 
		\begin{equation*}
		R_{0}\left(z\right)=\frac{1}{4\pi}\frac{e^{-\sqrt{-z}\left|x-y\right|}}{\left|x-y\right|}\label{eq:66},
		\end{equation*}
		where $\sqrt{-z}$ is the main branch of the square root.
		
		We will use $R_{0}\left(-1\right).$ It is clear that 
		\begin{equation*}
		\left(-\Delta+1\right)^{-1}\left(H+1\right)P_{c}=\left(I+R_{0}\left(-1\right)V\right)P_{c}\in\mathcal{B}\left(L^{\infty}\right)\cap B\left(L^{3,\infty}\right).\label{eq:67}
		\end{equation*}
		Therefore using the boundedness of the Mihlin multiplier, 
		\begin{equation*}
		\left(-\Delta+1\right)^{-1-i\sigma}\left(H+1\right)^{1+i\sigma}P_{c}\label{eq:68}
		\end{equation*}
		is bounded on $L^{p}$ with $3<p<\infty$ and on $L^{3,\infty}$
		with norm of at most exponential growth. 
		
		Also, by the boundedness of the Mihlin multiplier, 
		\begin{equation*}
		\left(-\Delta+1\right)^{-i\sigma}\left(H+1\right)^{i\sigma}P_{c}\label{eq:69}
		\end{equation*}
		is bounded on $L^{p}$ with $1<p<\infty$ with norms of at most exponential
		growth.
		
		Using the complex interpolation, for $0\leq s\leq1$, we can conclude
		that 
		\begin{equation*}
		\left(-\Delta+1\right)^{-s}\left(H+1\right)^{s}P_{c}\in\mathcal{B}\left(L^{p}\right)\label{eq:610}
		\end{equation*}
		with $\frac{3}{3-2s}<p<\infty$.
		
		Similarly, 
		\begin{equation*}
		\left(H+1\right)^{-1}P_{c}\left(-\Delta+1\right)=P_{c}\left(I+R_{0}\left(-1\right)V\right)^{-1}\in\mathcal{B}\left(L^{\infty}\right)\bigcap B\left(L^{3,\infty}\right).\label{eq:611}
		\end{equation*}
		Then the same argument as above gives us 
		
		\begin{equation*}
		\left(H+1\right)^{s}P_{c}\left(-\Delta+1\right)^{-s}\in\mathcal{B}\left(L^{p}\right)\label{eq:612}
		\end{equation*}
		with $\frac{3}{3-2s}<p<\infty$.
		
		The remaining estimates are obtained by duality.
	\end{proof}
	The above lemma enables us to interchange $-\Delta+1$
	and $H+1$ in the definition of the Sobolev space $H^{s}$
	with $-\frac{3}{2}<s<\frac{3}{2}$. To be more precise, we have
	\begin{lemma}
		\label{lem:sob}For $-\frac{3}{2}<s<\frac{3}{2}$, 
		\begin{equation*}
		H^{s}=\left\{ f\in\mathcal{S}':\,\left(-\Delta+1\right)^{\frac{s}{2}}f\in L^{2}\right\} =\left\{ f\in\mathcal{S}':\,\left(H+1\right)^{\frac{s}{2}}f\in L^{2}\right\} .\label{eq:613}
		\end{equation*}
	\end{lemma}


\begin{thebibliography}{ABC1}
		\bibitem[AAR]{aar} G.\ E.\ Andrews, R.\ Askey, R.\ Roy, \emph{Special Functions (Encyclopedia of Mathematics and its Applications)}, Cambridge University Press, 2001.
		\bibitem[Agmon]{Agmon} S.\ Agmon, \emph{Spectral properties of Schr\"odinger
			operators and scattering theory}, Ann. Scuola Norm. Sup. Pisa
		Cl. Sci. (4) 2 (1975), no.~2, 151--218.
		\bibitem[BS]{BS} M.\ Beals, W.\ Strauss, \emph{$L^{p}$ estimates for the wave equation with a potential}, Comm. Partial Differential Equations 18 (1993), no. 7-8, 1365--1397.
		\bibitem[Be]{BE} M.\ Beals, \emph{Optimal $L^{\infty}$ decay for solutions to the wave equation with a potential}, Comm. Partial Differential Equations 19 (1994), no. 7-8, 1319--1369.
		\bibitem[Bec]{Bec} M.\ Beceanu, \emph{ Dispersive estimates in $\mathbb{R}^3$ with threshold eigenstates and resonances}. Anal. PDE 9 (2016), no. 4, 813--858.
		\bibitem[BeGo]{becgol} M.\ Beceanu, M.\ Goldberg, \emph{Strichartz estimates and maximal operators for the wave equation in $\R^3$}, Journal of Functional Analysis (2014), Vol.\;266, Issue 3, pp.\;1476--1510.
		
		\bibitem[BeGo1]{becgol1} M.\;Beceanu, M.\;Goldberg, \emph{Spectral multipliers for Hamiltonians with scalar potential, I}, preprint.
		
		\bibitem[BeSch]{BeSch} M.\ Beceanu, W.\ Schlag, \emph{Structure formulas for wave operators}, Amer. J. Math. 142 (2020), no. 3, 751--807.
		\bibitem[BeL\"{o}]{bergh} J.\ Bergh, J.\ L\"{o}fstr\"{o}m, \emph{Interpolation Spaces: An Introduction}, Springer-Verlag, Grundlehren der mathematischen Wissenschaften, Vol.\;223, 1976.
		\bibitem[Che1]{C1} G.\ Chen, \emph{Strichartz estimates for wave equations with charge transfer Hamiltonians}, arXiv:1610.05226. (accepted by Memoirs of the AMS).
		\bibitem[Che2]{C2} G.\ Chen, \emph{Multisolitons for the defocusing energy critical wave equation with potentials}, Commun. Comm. Math. Phys. 364 (2018), no. 1, 45--82.
		\bibitem[Dan]{Dan} Piero D'Ancona, \emph{Kato smoothing and Strichartz estimates for wave equations with magnetic potentials}, Communications in Mathematical Physics (2015), Vol.\;335, pp.\;1--16.
		\bibitem[DaF]{DaF} P.\ D'Ancona,  L.\ Fanelli, \emph{Strichartz and smoothing estimates of dispersive equations with magnetic potentials}, Comm. Partial Differential Equations 33 (2008), no. 4--6.
		\bibitem[GoSch]{GoSch} M.\ Goldberg, W.\  Schlag, \emph{A limiting absorption principle for the three-dimensional Schr\"odinger equation with $L^p$ potentials}, Int. Math. Res. Not. 2004, no. 75, 4049--4071.
\bibitem[Hon]{hong} Y.\; Hong, \emph{A spectral multiplier theorem associated with a Schr\"odinger operator}, J.\; Fourier Anal.\; Appl. {\bf 22} (2016), no.\;3, 591--622.
		\bibitem[IMN]{IMN} S.\;Ibrahim, N.\;Masmoudi, K,\;Nakanishi, Scattering threshold for the focusing nonlinear Klein-Gordon equation. Anal. PDE 4 (2011), no. 3, 405--460.
		\bibitem[JLSchX]{JLSchX} H. Jia,  B.\ P.\ Liu, W.\ Schlag and G.\ X.\ Xu, \emph{ Global center stable manifold for the defocusing energy critical wave equation with potential}, to appear in American Journal of Math.
		\bibitem[KeTa]{KeTa} K.\ Keel, T.\ Tao, \emph{Endpoint Strichartz estimates}, Amer. J. Math. 120 (1998), no. 5, 955--980. 
		\bibitem[MNNO]{MNNO} S.\ Machihara, M.\ Nakamura, K.\ Nakanishi, T.\ Ozawa, \emph{Endpoint Strichartz estimates and global solutions for the nonlinear Dirac equation}, J. Funct. Anal. 219 (2005), no. 1, 1--20.
		\bibitem[MNNO2]{MNNO2}S.\;Machihara, K.\;Nakanishi, T.\;Ozawa, \emph{Small global solutions and the nonrelativistic limit for the nonlinear Dirac equation}, Rev. Mat. Iberoamericana 19 (2003), no. 1, 179--194.
		\bibitem[MSS]{MSS} B.\ Marshall, W.\ Strauss, S.\ Wainger.  \emph{$L^{p}-L^{q}$ estimates for the Klein-Gordon equation}, J. Math. Pures Appl. (9) 59 (1980), no. 4, 417--440.
		\bibitem[NaSch]{NaSch} K.\ Nakanishi, W.\ Schlag,  \emph{Invariant manifolds and dispersive Hamiltonian evolution equations}, Zurich Lectures in Advanced Mathematics. European Mathematical Society (EMS), Z\"urich, 2011. vi+253 pp. 
		\bibitem[ReSi4]{ReSi4} M.\ Reed, B.\ Simon, \emph{Methods of modern mathematical physics. IV. Analysis of operators.} Academic Press, New York-London, 1978. xv+396 pp.
		\bibitem[Ste]{Stein} E. Stein, \emph{Harmonic Analysis}, Princeton University Press, Princeton, 1994
	\end{thebibliography}
\end{document}